\documentclass[11pt,reqno]{amsart}
\usepackage{amsmath, latexsym, amsfonts, amssymb,amsthm, amscd,epsfig,enumerate}
\usepackage{subfigure}
\advance\hoffset -.75cm

\usepackage{amsfonts}
\usepackage{latexsym}
\usepackage{amsmath}
\usepackage{amssymb}
\usepackage[usesnames]{xcolor}

\definecolor{refkey}{gray}{.75}
\definecolor{labelkey}{gray}{.75}

\newcommand{\C}{\mathbb C}

\newcommand{\Z}{\mathbb Z}
\newcommand{\N}{\mathbb N}
\newcommand{\pr}{\mathbb P}
\newcommand{\E}{\mathbb E}

\newcommand{\diff}{\rm d}
\newcommand{\var}{\mathrm{var}}

\newcommand{\eps}{\varepsilon}

\newcommand{\cL}{\mathcal{L}}

\newcommand{\ident}{{\mathchoice {\rm 1\mskip-4mu l} {\rm 1\mskip-4mu l}
{\rm 1\mskip-4.5mu l} {\rm 1\mskip-5mu l}}}

\renewcommand{\thesubfigure}{\arabic{subfigure}}
\makeatletter
\renewcommand{\@thesubfigure}{\tiny Figure \thesubfigure: \space}
\renewcommand{\p@subfigure}{}
\makeatother

\newtheorem{teo}{Theorem}[section]
\newtheorem{lem}[teo]{Lemma}
\newtheorem{cor}[teo]{Corollary}
\newtheorem{rem}[teo]{Remark}
\newtheorem{pro}[teo]{Proposition}
\newtheorem{defn}[teo]{Definition}
\newtheorem{exmp}[teo]{Example}

\newtheorem{assump}[teo]{Assumption}
\newtheorem{step}{Step}

\begin{document}

\title[Approximating discrete-time BRWs]
{Survival, extinction and approximation of discrete-time branching random walks}

\author[F.~Zucca]{Fabio Zucca}
\address{F.~Zucca, Dipartimento di Matematica,
Politecnico di Milano,
Piazza Leonardo da Vinci 32, 20133 Milano, Italy.}
\email{fabio.zucca\@@polimi.it}

\begin{abstract}
We consider a general discrete-time branching random walk on a countable
set $X$. We relate local, strong local and global survival with suitable inequalities
involving the first-moment matrix $M$ of the process. In particular we prove
that, while the local behavior is characterized by $M$,
the global behavior cannot be completely described in terms of properties involving
$M$ alone. Moreover we show that locally surviving branching random walks can be approximated by
sequences of spatially confined and stochastically dominated branching random walks which eventually survive locally if the
(possibly finite) state space is large enough. An analogous result can be achieved
by approximating a branching random walk by a sequence of multitype contact processes and allowing
a sufficiently large number of particles per site. We compare these results with the
ones obtained in the continuous-time case and we give some examples and counterexamples.
\end{abstract}

\date{}
\maketitle


\baselineskip .6 cm

\section{Introduction}
\label{sec:intro}
\setcounter{equation}{0}

The theory of branching random walks (BRWs from now on) has a long history dating back
to the earlier works on discrete-time branching processes (see \cite{cf:GW1875} for the original work of Galton
and Watson and \cite{cf:AthNey, cf:Harris63}). 
In the last 20 years much effort has been put in the study of continuous-time BRWs
(see \cite{ cf:HuLalley, cf:Ligg1, cf:Ligg2, cf:MadrasSchi, cf:MountSchi} just to name a few).
Among all the topics which have been studied there is the distinction between local
and global survival (see for instance \cite{cf:PemStac1, cf:Stacey03, cf:BZ, cf:BZ2})
and the relation between multitype contact processes and BRWs (see \cite{cf:BZ3}); besides,
some papers have 
explored the subject of continuous-time BRWs
on random environments (see for instance \cite{cf:GDH92}).

Discrete-time BRWs have been studied initially as a natural generalization of branching
processes (see \cite{cf:AthNey, cf:Big1977, cf:Big1978, cf:BigKypr97, cf:BigRah05, cf:Harris63}
just to mention a few).
In recent years there has been a growing interest on discrete-time BRWs on deterministic
graphs and on random environments (see \cite{cf:CMP98, cf:GMPV09, cf:DHMP99, cf:MP00, cf:MP03, cf:M08}).
It is well-known that any continuous-time BRW admits a discrete-time counterpart with the same behavior (see
Section~\ref{subsec:continuous}), thus, under a certain point of view, discrete-time BRWs
generalize continuous-time BRWs; this is the analogous of the construction of the jump chain
associated to a continuous-time Markov chain.
Hence discrete-time BRWs help to understand continuous-time processes as well
(see for instance \cite{cf:BZ, cf:BZ2}).

There are many reasons for studying fairly simple and non-interacting models such as BRWs.
One of them is the well-known connection with percolation theory.
Besides, they are fundamental tools in the journey to understanding
more sophisticated models: for instance, they are frequently used as comparison to prove survival or
extinction of different particle systems using the well-known technique called \textit{coupling}.
This last reason justifies our choice of studying \textit{truncated BRWs} (see
Sections~\ref{subsec:1} and \ref{sec:truncated}).

One of the differences between the discrete-time and the continuous-time settings
is that in the second one the Markov property identifies a unique family of distributions
for the time intervals between births and deaths and thus a unique family of laws
for the random number of children of each particle.
Hence, even if a continuous-time process seems to be a more realistic picture in view
of applications, on the other hand it appears as a strong restriction if compared to
the wide choice of reproduction laws that one can consider in the discrete-time
case. Moreover, it is customary in the continuous-time setting, to study
a one-parameter family of BRWs simultaneously in order to compute the intervals
corresponding to different behaviors (see Section~\ref{subsec:continuous}
for details). In contrast, the main results for discrete-time BRWs deal with one
single process; nevertheless this is not a serious restriction and it is easy
to see that the results on survival which are known for continuous-time BRWs can
be obtained by applying the results for discrete-time BRWs to their discrete-time counterparts.
In this sense the results of this paper ``generalize'' those of \cite{cf:BZ, cf:BZ2};
we give more details in the outline of the paper below.
We study the survival of a BRW in the first part of this paper.
As for the topic of the second part, namely the approximation
of a BRW, we see that, while the results on the spatial approximation (see Section~\ref{sec:spatial})
generalize those of \cite[Section 3]{cf:BZ3}, the results about the approximation by
truncated BRWs (see Section~\ref{sec:truncated}) cannot be seen as a generalization
of the analogous results of \cite[Section 5]{cf:BZ3}. This is due to the fact
that the discrete-time counterpart of a continuous-time truncated BRW is not
a discrete-time truncated BRW. Indeed note that in the definition of a truncated BRW
(see Section~\ref{subsec:1} and \cite[Section 2]{cf:BZ3}) the time scale is essential.

The aim of this paper is threefold: we want to study the global, local and strong local behavior of discrete-time BRWs,
the possibility of approximating BRW with a sequence of ``spatially confined'' and stochastically dominated BRWs and, finally,
the approximation of a BRW by means of a sequence of truncated BRWs
which are, essentially, \textit{multitype contact processes}.
The results of this paper generalize
those of \cite{cf:BZ, cf:BZ2, cf:BZ3} not only because the class of discrete-time
BRWs extends the class of continuous-time BRWs but also since some of the theorems
are stronger and require weaker hypotheses.

Here is the outline of the paper.
In Section~\ref{sec:discrete-continuous} we define discrete-time BRWs and discuss their main
properties. This is a natural generalization of the class of \textit{multitype Galton--Watson
branching process} (see \cite[Section II.2]{cf:Harris63}) similar to those introduced in other papers
(see \cite{cf:BZ, cf:BZ2, cf:GMPV09, cf:M08} for some recent references).
In Section~\ref{subsec:continuous} we briefly introduce continuous-time BRWs and
we construct their discrete-time counterparts. In Section~\ref{sec:technical} we give the technical
definitions and we state some basic results. Section~\ref{subsec:localglobal} is devoted to the
study of local and global survival. The main result (Theorem~\ref{th:equiv1}) characterizes
local survival by means of the first-moment matrix $M$ of the process (see Section~\ref{subsec:discrete}),
and global survival
using a possibly infinite-dimensional generating function associated to the BRW.
This theorem generalizes \cite[Theorems 4.1, 4.3 and 4.7]{cf:BZ2}. An independent
proof of Theorem~\ref{th:equiv1}(1) appeared in \cite[Theorem 2.4]{cf:M08}.
A similar, but weaker, result for a limited class of continuous-time BRWs
can be found in \cite{cf:PemStac1} and, for the whole class of continuous-time
BRWs, in \cite{cf:BZ, cf:BZ2}.
 We show
that, in general, global survival cannot be characterized in terms of the first-moment
matrix alone (see Example~\ref{exm:noext}), nevertheless some functional inequalities involving
only the first-moment matrix $M$ must hold in case of global survival (see Theorem~\ref{th:equiv1}).
We introduce a class of fairly regular BRWs, which includes BRWs on quasi-transitive graphs
and BRWs on regular graphs, for which we can give a complete characterization of global survival
in terms of the matrix $M$. Roughly speaking, the regularity that we require is the possibility
of mapping these BRWs into multitype Galton--Watson processes for which a complete
characterization of the global survival is known (see for instance \cite{cf:Harris63} and
Theorem~\ref{th:equiv1}(5)).
In Section~\ref{subsec:stronglocal} we show that local survival
can be described, as the global one, by means of an infinite-dimensional generating function;
this is the key to discuss strong local survival.
This kind of survival was already studied, for instance, in \cite{cf:GMPV09} for a BRW
on a random environment on $\Z$. Here,
using a different technique, we are able to treat a large class of transitive BRWs on deterministic
graphs.
In Section~\ref{sec:spatial} we first generalize a Theorem due to Sarymshakov and Seneta
(see \cite[Theorem 6.8]{cf:Sen}) and then we use this result (Theorem~\ref{th:genseneta}) to obtain an approximation
of a general BRW, which is not necessarily irreducible,
by means of a sequence of spatially confined BRWs (Theorem~\ref{th:spatial});
this last one is a generalization of \cite[Theorem 3.1]{cf:BZ3}.
Here we obtain, as a particular case, that
if we have a surviving process, then by confining it
to a sufficiently large (possibly finite and not necessarily connected) proper subgraph
the resulting BRW survives as well;
this result was already known for irreducible BRWs confined to connected subgraphs.
At the end of the section we give some examples and counterexamples.
Section~\ref{sec:truncated} deals with the approximation of the BRW with a sequence
of truncated BRWs. 
The key to obtain such a result
is the comparison of our process with a suitable oriented percolation (as explained in Section~\ref{sec:roadmap}).
The strategy is then applied to some classes of regular BRWs in Theorem~\ref{th:main} (concerning
local behavior) and Theorem~\ref{th:zdrift} (concerning global behavior). Finally in Section~\ref{sec:open}
we briefly discuss some open questions and possible future developments.

\section{The dynamics: discrete and continuous time}
\label{sec:discrete-continuous}


\subsection{Discrete-time branching random walks}
\label{subsec:discrete}

We start with the construction of a generic discrete-time BRW (see also \cite{cf:BZ2} where it is
called \textit{infinite-type branching process}) on a set $X$ which is
at most countable. To this aim we
consider a general family $\mu=\{\mu_x\}_{x \in X}$
of probability measures on 
$S_X:=\{f:X \to \N:\sum_yf(y)<\infty\}$.
The updating rule is
the following: a particle at a site $x\in X$ lives one unit of time,
then, with probability $\mu_x$, a function $f \in S_X$ is chosen
and the original particle is replaced by $f(y)$ particles at
$y$, for all $y \in X$; this is done independently for all the particles.

Here is another equivalent dynamics:
define the function $H:S_X \rightarrow \N$ as $H(f):=\sum_{x \in X} f(x)$
and denote by $\rho_x$ the measure on $\N$ defined by
$\rho_x(\cdot):=\mu_x(H^{-1}(\cdot))$; this is the law of the random number of children
of every particle living at $x$. For each particle, independently, we pick a number $n$ at random,
according to the law $\rho_x$, and then we choose a function $f \in H^{-1}(n)$ with probability
$\mu_x(f)/\rho_x(n)\equiv\mu_x(f)/\sum_{g \in H^{-1}(n)}\mu_x(g)$ and, again,
we replace the particle at $x$ with $f(y)$ particles at $y$ (for all $y \in X$).

More precisely, given a family $\{ f_{i,n,x}\}_{i,n \in \N, x \in X}$
of independent $S_X$-valued random variable such that, for every $x \in X$, $\{f_{i,n,x}\}_{i,n \in \N}$ have the common law $\mu_x$, then
the discrete-time BRW $\{\eta_n\}_{n \in \N}$ is defined
iteratively as follow
\begin{equation}\label{eq:evolBRW}
\eta_{n+1}(x)=\sum_{y \in X} \sum_{i=1}^{\eta_n(y)} f_{i,n,y}(x)=
\sum_{y \in X} \sum_{j=0}^\infty \ident_{\{\eta_n(y)=j\}} \sum_{i=1}^{j} f_{i,n,y}(x)
\end{equation}
starting from an initial condition $\eta_0$. We denote the BRW by $(X,\mu)$; the
initial value will be clearly indicated each time.


Denote by
$m_{xy}:=\sum_{f\in S_X} f(y)\mu_x(f)$
 the expected number of particles from $x$ to $y$
(that is, the expected number of children that a particle living
at $x$ can send to $y$)
 and
suppose that $\sup_{x \in X} \sum_{y \in X} m_{xy}<+\infty$; most of the results
of this paper still hold without this hypothesis, nevertheless
it allows us to avoid dealing with an infinite expected number of offsprings.
Note that $\sum_{y \in X} m_{xy} = \sum_{n \ge 0} n \rho_x(n)=:\bar \rho_x$.

We denote by
 $M=(m_{xy})_{x,y \in X}$ the \textit{first-moment} matrix and by $m^{(n)}_{xy}$
the entries of the matrix $M^n$.
We call \textit{diffusion matrix} the matrix $P$ with entries $p(x,y)=m_{xy}/\bar \rho_x$.

From equation~\eqref{eq:evolBRW},
it is straightforward to prove that the expected number of particles,
starting from an initial condition $\eta_0$,
satisfies the recurrence equation $\E^{\eta_0}(\eta_{n+1}(x))=(\E^{\eta_0}(\eta_{n})M)(x)=\sum_{y \in X}m_{yx} \E^{\eta_0}(\eta_n(y))$
hence
\begin{equation}\label{eq:evolexpectednumber}
\E^{\eta_0}(\eta_n(x)) = \sum_{y \in X} m^{(n)}_{yx} \eta_0(y).
\end{equation}

Moreover, the family of probability measures, $\{\mu_x\}_x$ induces in a natural way a graph structure on $X$ that
we denote by $(X,E_\mu)$ where $E_\mu:=\{(x,y):m_{xy}>0\}\equiv\{(x,y): \exists f \in S_X, \mu_x(f)>0, f(y)>0\}$.
Roughly speaking, $(x,y)$ is and edge if and only if a particle living at $x$ can
send a child at $y$ with positive probability (from now on \textit{wpp}).
We say that there is a path from $x$ to $y$, and we write $x \to y$, if it is
possible to find a finite sequence $\{x_i\}_{i=0}^n$ such that $x_0=x$, $x_n=y$ and $(x_i,x_{i+1}) \in E_\mu$
for all $i=0, \ldots, n-1$. If $x \to y$ and $y \to x$ we write $x \rightleftharpoons y$.

Recall that the matrix $M=(m_{xy})_{x,y \in X}$ is said to be \textit{irreducible} if and only if
the graph $(X,E_\mu)$ is \textit{connected}, otherwise we call it \textit{reducible}. We denote by $\mathrm{deg}(x)$
the degree of a vertex $x$, that
is, the cardinality of the set $\{y\in X: (x,y) \in E_\mu
\}$.

The colony can survive in different ways: we say that the
colony survives locally wpp at $y \in X$ starting from $x \in X$
if
\[
\pr^{\delta_x}(\limsup_{n \to \infty} \eta_n(y)>0)>0;
\]
we say that it survives globally wpp starting from $x$ if
\[
\pr^{\delta_x} \Big (\sum_{w \in X} \eta_n(w)>0, \forall n \in N \Big )>0.
\]
Moreover, following \cite{cf:GMPV09}, we say that the
there is strong local survival wpp at $y \in X$ starting from $x \in X$
if
\[
\pr^{\delta_x}(\limsup_{n \to \infty} \eta_n(y)>0)=
\pr^{\delta_x} \Big (\sum_{w \in X} \eta_n(w)>0, \forall n \in N \Big )>0.
\]
From now on when we talk about survival, ``wpp'' will be tacitly understood.
Often we will say simply that local survival occurs ``starting from $x$'' or ``at $x$'':
in this case we mean that $x=y$.

Clearly local survival implies global survival and,
if $x \to y$ 
then local survival at $x$ implies local survival
at $y$ starting from $x$. Analogously,
if $x \to y$ then
global survival starting from $y$ implies global survival starting from $x$.
Moreover if $x \rightleftharpoons y$ then local (resp.~global) survival
starting from $x$ is equivalent to local (resp.~global) survival starting from $y$.
In particular, if $M$ is irreducible then the process survives locally (resp.~globally) at one vertex if and
only if it survives locally (resp.~globally) at every vertex.

\begin{assump}\label{assump:1}
We assume henceforth that for all $x \in X$ there is a vertex $y \rightleftharpoons x$ such that
$\mu_y(f: \sum_{w \rightleftharpoons y} f(w)=1)<1$,
 that is, in every equivalence class (with respect to $\rightleftharpoons$)
there is at least one vertex where a particle
can have a number of children different from one wpp.
\end{assump}

\begin{rem}\label{rem:assump1}
The previous assumption guarantees that the restriction of the BRW
to an equivalence class is \textit{nonsingular} (see \cite[Definition II.6.2]{cf:Harris63}).
%
There is a technical reason behind the previous assumption: if we consider
the classical Galton--Watson branching process, that is, $X:=\{x\}$ is a singleton
and $\mu_x$ can be considered as a probability measure on $\mathbb N$, then it is well-known that
\begin{itemize}
 \item if $\mu_x(1)=1$ then $m_{xx}=1$ and there is survival with probability $1$;

\item if $\mu_x(1)<1$ then there is survival wpp if and
only if $m_{xx}>1$.
\end{itemize}
Hence the condition  $m_{xx}>1$ is equivalent to survival under
Assumption~\ref{assump:1}. This will be used implicitly in Theorem~\ref{th:equiv1}(1) and (5)
(but it is not needed, for instance, in Theorem~\ref{th:equiv1}(2), (3) and (4)).
\end{rem}


\medskip

A particular, but meaningful, subclass of discrete-time BRWs is described by the following updating rule:
a particle at site $x$ lives one unit of time
and is replaced by
a random number of children, with law $\rho_x$.
The children are dispersed independently on the
sites of the graph, according to a stochastic matrix $P$.
Note that this rule is a particular case of the
general one, since here one simply chooses 
\begin{equation}\label{eq:particular1}
\mu_x(f)=\rho_x \left (\sum_y f(y) \right )\frac{\sum_y f(y)!}{\prod_y f(y)!} \prod_y (p(x,y))^{f(y)},
\quad \forall f \in S_X.
\end{equation}
\smallskip
Clearly in this case
 the expected number of children at $y$ of a particle
living at $x$ is
\begin{equation}\label{eq:meanparticular}
 m_{xy}=p(x,y) \bar \rho_x.
\end{equation}

\subsection{Continuous-time branching random walks}
\label{subsec:continuous}

Continuous-time BRWs have been studied extensively by many authors;
in this section we make use of a natural correspondence between continuous-time BRWs and
discrete-time BRWs which preserves both local and global behaviors.

In continuous time each particle has an exponentially distributed
random lifetime with parameter 1. The breeding mechanisms can
be regulated by putting on each edge $(x,y)$ and for each particle at $x$,
a clock with $Exp(\lambda k_{xy})$-distributed intervals (where $\lambda>0$),
each time the clock
rings the particle breeds in $y$.
Equivalently one can associate to each particle at $x$
a clock with $Exp(\lambda k(x))$-distributed intervals ($k(x)=\sum_y k_{xy}$):
each time the clock rings the particle breeds and the offspring is placed
at random according to a stochastic matrix $P$
(where $p(x,y)=k_{xy}/k(x)$).

The formal construction of a BRW in continuous time is based on
the action of a semigroup with infinitesimal generator
\begin{equation}\label{gengen}
\cL f (\eta):=  \sum_{x \in X} \eta(x)\Big (\partial_x^- f(\eta) + \lambda \sum_{y\in X} \,
k_{xy}\,\partial_y^+ f(\eta) \Big ),
\end{equation}
where $\partial_x^\pm f(\eta):=f(\eta\pm\delta_x)-f(\eta)$.


Every continuous-time BRW
has a discrete-time counterpart and they both survive or both die
(locally or globally);
here is the construction.
The initial particles represent the generation $0$ of the discrete-time BRW;
the generation $n+1$ (for all $n \ge 0$) is obtained by
considering the children of all the particles of generation $n$
(along with their positions).
Clearly the progenies of the original continuous-time BRW and of its discrete-time counterpart
are both finite (or both infinite) at the same time.
In this sense the theory of continuous-time BRWs, as long as we
are interested in the probability of survival (local, strong local and global),
is a particular case of the theory of discrete-time BRWs.

Elementary calculations show that each particle living at $x$, before dying,
has a random number of offsprings given by equation~\eqref{eq:particular1} where
\begin{equation}\label{eq:counterpart}
\rho_x(i)=\frac{1}{1+\lambda k(x)} \left ( \frac{\lambda k(x)}{1+\lambda k(x)} \right )^i, \qquad
p(x,y)=\frac{k_{xy}}{k(x)},
\end{equation}
and this is the law of the discrete-time counterpart. Using equation~\eqref{eq:meanparticular},
it is straightforward to show that $m_{xy}=\lambda k_{xy}$.
From equation~\eqref{eq:counterpart} we have that, for any $\lambda>0$, the discrete-time
counterpart satisfies Assuption~\ref{assump:1}.


\smallskip

Given $x_0 \in X$, two critical parameters are associated to the
continuous-time BRW: the global 
survival
critical parameter $\lambda_w(x_0)$ and the local 
survival one $\lambda_s(x_0)$.
They are defined as
\begin{equation}\label{eq:criticalparameters}
\begin{split}
\lambda_w(x_0)&:=\inf\{\lambda>0:\,\pr^{\delta_{x_0}}\left(\exists t:\eta_t=\mathbf{0}\right)<1\}\\
\lambda_s(x_0)&:=\inf\{\lambda>0:\,\pr^{\delta_{x_0}}\left(\exists \bar t:\eta_t(x_0)=0,\,\forall t\ge\bar t\right)<1\},
  \end{split}
\end{equation}
where $\mathbf{0}$ is the configuration with no particles at all
sites and $\pr^{\delta_{x_0}}$ is the law of the process which
starts with one individual in $x_0$. If the graph $(X, E_{\mu})$
is connected then these values do not depend on the initial
configuration, provided that this configuration is finite (that
is, it has only a finite number of individuals), nor on the choice
of $x_0$. See \cite{cf:BZ} and \cite{cf:BZ2} for a more detailed
discussion on the values of $\lambda_w(x_0)$ and $\lambda_s(x_0)$.


\section{Technical definitions}\label{sec:technical}

In this section we give some technical definitions
and we state some basic facts which are widely used in the
rest of the paper.

\subsection{Reproduction trails}\label{subsec:trails}

A fundamental tool which is useful throughout the whole paper is
the reproduction trail;
this allows us to give an alternative construction of the BRW.
We fix an injective map $\phi:X\times X \times \Z \times \N \rightarrow \N$.
Let the family $\{ f_{i,n,x}\}_{i,n \in \N, x \in X}$ be as in Section~\ref{subsec:discrete} and let
$\eta_0$ be the initial value. For any fixed realization of the process we call \textit{reproduction trail}
to $(x,n) \in X \times \N$ a sequence
\begin{equation}\label{eq:trail}
(x_0,i_0,1),(x_1,i_1,j_1), \ldots ,(x_{n},i_{n}, j_{n})
\end{equation}
such that $-\eta_0(x_0) \le i_0 <0$, $0 < j_l \le f_{i_{l-1}, l-1, x_{l-1}}(x_l)$
and $\phi(x_{l-1}, x_l, i_{l-1}, j_l)=i_l$, where $0 <l \le n$.
The interpretation  is the following:
$i_n$ is the identification number of the particle, which
lives at $x_n$ at time $n$ and is the $j_n$-th offspring
of its parent. The sequence $\{x_0, x_1, \ldots, x_n\}$ is
the path induced by the trail (sometimes, we say
that the trail is based on this path).
Given any element $(x_l, i_l, j_l)$ of the trail~\eqref{eq:trail}, we say that the particle identified
by $i_n$ is a descendant of generation $n-l$ of the particle identified by $i_l$ and the trail
joining them is $(x_l,i_l,j_l), \ldots ,(x_{n},i_{n}, j_{n})$. We say also that the trail
of the particle $i_n$ is a prolongation of the trail of the particle $i_l$.

Roughly speaking the trail represents the past history
of each single particle back to its original ancestor, that is, the one living at time $0$; we note that from the couple
$(n, i_n)$, since the map $\phi$ is injective, we can trace back the entire genealogy of the particle.
The random variable $\eta_n(x)$ can be alternatively defined as the number of reproduction trails to
$(x,n)$.
This construction does not coincide with the one
induced by the equation~\eqref{eq:evolBRW} but the resulting processes have the same laws.

Finally, when $\mu_x$ does not depend on $x \in X$, it is worth mentioning another possible construction of the process as a
Markov chain indexed by trees (see for instance \cite[Section 5.C]{cf:Woess09}).

\subsection{Generating functions}\label{subsec:genfun}

Later on we will need some generating functions, both
1-dimensional and
infinite dimensional.
Define $T^n_x:=\sum_{y \in X}m^{(n)}_{xy}$ and
$\varphi^{(n)}_{xy}:=\sum_{x_1,\ldots,x_{n-1} \in X \setminus\{y\}} m_{x x_1} m_{x_1 x_2} \cdots m_{x_{n-1} y}$
(by definition $\varphi^{(0)}_{xy}:=0$ for all $x,y \in X$).
$T^n_x$ is the expected number of particles alive at time $n$ when the initial state is a single
particle at $x$.
The interpretation of $\varphi^{(n)}_{xy}$ is more subtle. It plays in the BRW theory
the same role played the first-return probabilities in random walk theory
(see \cite[Section 1.C]{cf:Woess09}). Roughly speaking, $\varphi^{(n)}_{xy}$ is
the expected number of particles alive at $y$ at time $n$
when the initial state is just one particle at $x$ and the
process behaves like a BRW except that every particle reaching
$y$ at any time $i <n$ is immediately killed (before breeding).

Let us consider the following family of 1-dimensional generating functions
(depending on $x,y \in X$)
\[
\begin{split}
\Gamma(x,y|\lambda)&:=\sum_{n =0}^\infty m^{(n)}_{xy} \lambda^n,
\qquad \Phi(x,y|\lambda):=\sum_{n =1}^\infty \varphi_{xy}^{(n)} \lambda^n.
\end{split}
\]
%
It is easy to prove that $\Gamma(x,x|\lambda)= \sum_{i \in \N} \Phi(x,x|\lambda)^i$ for all $\lambda>0$, hence
\begin{equation}\label{eq:genfun1}
\Gamma(x,x|\lambda)= \frac{1}{1-\Phi(x,x|\lambda)},
\qquad \forall \lambda \in \C: |\lambda|< \left (\limsup_{n \in \N} \sqrt[n]{m^{(n)}_{xy}} \right )^{-1},
\end{equation}
and we have that $\left (\limsup_{n \in \N} \sqrt[n]{m^{(n)}_{xy}} \right )^{-1}=\max\{ \lambda 
\in {\mathbb R}:\Phi(x,x|\lambda)\leq 1\}$
for all $x \in X$. 
In particular $\Phi(x,x|1) \le 1$ if and only if
$\limsup_{n \in \N} \sqrt[n]{m^{(n)}_{xy}} \le 1$.

To the family $\{\mu_x\}_{x \in X}$ we can associate another generating function $G:[0,1]^X \to [0,1]^X$
which can be considered as an infinite dimensional power series (see also \cite[Section 3]{cf:BZ2}). More precisely,
for all $z \in [0,1]^X$ the function $G(z) \in [0,1]^X$ is defined as follows
\begin{equation}
\label{eq:genfun}
G(z|x):= \sum_{f \in S_X} \mu_x(f) \prod_{y \in X} z(y)^{f(y)}.
\end{equation}
Note that $G$ is continuous with respect to the \textit{pointwise convergence topology} of $[0,1]^X$  and nondecreasing
with respect to the usual partial order of $[0,1]^X$ (see \cite[Sections 2 and 3]{cf:BZ2} for further details);
note that $v <w$ means that $v(x) \le w(x)$ for all $x \in X$ and $v(x_0)<w(x_0)$ for some $x_0 \in X$.
All these generating functions will be useful for instance in Section~\ref{subsec:localglobal} to prove
Theorem~\ref{th:equiv1} and to discuss Examples~\ref{exm:noext} and \ref{exm:4.5}.
In Section~\ref{subsec:stronglocal} more properties of $G$ will be
established in order to prove some results related to local survival and strong local survival.

Note that the generating function $G$ can be explicitly computed, for instance, if
equation~\eqref{eq:particular1} holds. Indeed in this case it is straightforward
to show that
$G(z|x)=F(Pz(x))$
where $F(y)=\sum_{n=0}^\infty \rho(n)y^n$ and $Pz(x)=\sum_{y \in X} p(x,y)z(y)$.
In particular if $\rho(n)=\frac{1}{1+\bar \rho_x} (\frac{\bar \rho_x}{1+\bar \rho_x} )^n$
(for instance if we are dealing with the discrete-time counterpart of a continuous-time BRW, see
equation~\eqref{eq:counterpart}), we have
$G(z|x)=\frac{1}{1+\bar \rho_x(1-Pz(x))}$,
that is,
\begin{equation}\label{eq:Gcontinuous}
G(z)= \frac{\mathbf{1}}{\mathbf{1}+M(\mathbf{1}-z)}
\end{equation}
where $\mathbf{1}(x):=1$ for all $x \in X$, the ratio is to be intended as coordinatewise and
$Mv(x):= \sum_{y \in X} m_{xy}v(y)$ for all $v \in [0,1]^X$; in this case
$m_{xy}$ is given by equation~\eqref{eq:meanparticular}.

\subsection{Coupling}\label{subsec:1}

The family of BRWs can be extended to the more general class of \textit{truncated BRWs} where
a maximum of $m \in \N \cup \{\infty\}$ particles per site are allowed; we denote
this process as a BRW$_m$. The general dynamics is given by the following
recursive relation
\begin{equation}\label{eq:evolBRWm}
\eta^m_{n+1}(x)=m \wedge \sum_{y \in X} \sum_{i=1}^{\eta^m_n(y)} f_{i,n,y}(x)=
m \wedge \sum_{y \in X} \sum_{j=0}^\infty \ident_{\{\eta^m_n(y)=j\}} \sum_{i=1}^{j} f_{i,n,y}(x).
\end{equation}
Clearly the BRW$_\infty$ is the usual BRW and the BRW$_1$ is a \textit{contact process}.

In the following sections we want to compare two (or more)
truncated BRWs.
What we are going to state is a discrete-time analogous of a well-known
technique called \textit{coupling}. Roughly speaking, given two
processes $\{\eta_n\}_n$ and $\{\xi_n\}_n$, under certain conditions one can
find two possibly different processes  $\{\eta^\prime_n\}_n$ and $\{\xi^\prime_n\}_n$ with the same finite-dimensional
distribution as the original ones and such that
if $\eta_0^\prime \ge \xi_0^\prime$ then $\eta_n^\prime \ge \xi_n^\prime$ for all $n \in \N$.

The condition that we use is the following:
suppose we have two families $\mu=\{\mu_x\}_{x \in X}$ and
$\nu=\{\nu_x\}_{x \in X}$ such that
$\mu_x(F_x^{-1}(\cdot))=\nu_x(\cdot)$ for all $x\in X$ and for
some family of 
functions $\{F_x\}_{x\in X}$ such that
$F_x:\textrm{supp}(\mu_x)\rightarrow \textrm{supp}(\nu_x)$ and
$F_x(f) \le f$ for all $f \in \textrm{supp}(\mu_x)$.
The meaning of the maps $F_x$ is as follows:
given any possible offspring outcome $g$ of a particle living at $x$
(breeding according to the law $\nu_x$) there is a
set of outcomes which occurs with the same probability, namely $F_x^{-1}(g)$,
for a particle living at $x$ (breeding according to the law $\mu_x$);
furthermore the number of newborn particles in the first case is pointwise not larger than
the number of newborns in the second one. This suggests
that the BRW $(X, \mu)$, in some sense,
dominates the BRW $(X, \nu)$.
Under the previous condition, given $k \le m \le \infty$,
it is possible to construct a process $\{(\eta^m_n,\xi^k_n)\}_{n\in\N}$ such that
\begin{enumerate}
 \item
$\{\eta^m_n\}_{n\in\N}$ is
a BRW$_m$ behaving according to $\mu$;
\item
$\{\xi^k_n\}_{n \in \N}$ is
a BRW$_k$ behaving according to $\nu$;
\item
$\eta^m_0 \ge \xi^k_0$ implies $\eta^m_n \ge \xi^k_n$ for all $n\in\N$ a.s.
\end{enumerate}

To check (3), note that,
for any $x \in X$, given the family of random variables $\{f_{i,n,x}\}_{i,n \in \N}$
(with law $\mu_x$) then $\{F_x(f_{i,n,x})\}_{i,n \in \N}$ are iid with
common law $\nu_x$. Whence
the evolution equation of $\{\eta^m_n\}_{n \in \N}$ is \eqref{eq:evolBRWm} and,
similarly, $\{\xi^k_n\}_{n\in \N}$ satisfies
\begin{equation}\label{eq:evol2}
\xi^k_{n+1}(x)=k \wedge \sum_{y \in X} \sum_{i=1}^{\xi_n(y)} F_x \circ f_{i,n,y}(x),
\end{equation}
whence (3) follows by induction using equations \eqref{eq:evolBRWm} and \eqref{eq:evol2}.
A typical choice for the family of functions $\{F_x\}_{x \in X}$ is
$F_x(f):=f|_{Y}$ (where $Y\subseteq X$) which can be seen as a
(truncated) BRW restricted to $Y$, that is, all the offsprings
sent outside $Y$ are killed.

This procedure of comparison is called \textit{coupling}
between $\{\eta^m_n\}_{n \in \N}$ and $\{\xi^k_n\}_{n \in \N}$.
We note that if $\{\eta^m_n\}_{n \in \N}$ dies out locally (resp.~globally)
a.s.~then $\{\xi^k_n\}_{n \in \N}$ dies out locally (resp.~globally)
a.s.
More generally a 
coupling
between $\{\eta^m_n\}_{n \in \N}$ and $\{\xi^k_n\}_{n \in \N}$
is a choice of a common law $\{\zeta_{x}\}_{x \in X}$ for the process
$\{(\eta^m_n,\xi^k_n)\}_{n \in \N}$ such that $\zeta_{x}((f,g): f \ge g,\, f,g \in S_X) =1$ for all $x\in X$ and
$\sum_{g \in S_X}  \zeta_{x}((f,g))=\mu_x(f)$, $\sum_{f \in S_X}  \zeta_{x}((f,g))=\nu_x(g)$. In many
situations this construction of $\{\zeta_{x}\}_{x \in X}$ can be carried out effortlessly.

\section{Local, strong local and global survival}
\label{sec:localglobal}

\subsection{Local and global survival}
\label{subsec:localglobal}

Consider the discrete-time BRW generated by the family $\mu=\{\mu_x\}_x$
of probabilities and
suppose now that the process starts with one particle at $x_0$, hence
$\eta_0=\delta_{x_0}$. In this section we want to find conditions for
global, local and strong local survival.
Recall that if $X$ is finite then local survival is equivalent to global
survival: this is trivial for an irreducible matrix $M$; in the general
case global survival, starting from $x_0$, is equivalent to
local survival at some $y \in X$ such that $x_0 \to y$ (the same arguments of
\cite[Remark 4.4]{cf:BZ2} apply).

\begin{teo}\label{th:equiv1}
Let $(X,\mu)$ be a BRW.
\begin{enumerate}
\item There is
local survival starting from $x_0$ if and only if $\limsup_{n\to\infty}\sqrt[n]{m^{(n)}_{x_0x_0}}>1$.
\item There is
global survival starting from $x_0$ if and only if there exists $z\in [0,1]^X$, $z(x_0)<1$,
such that $G(z|x)\le z(x)$, for all $x$.
\item
If there is global survival starting from $x_0$, then
there exists $v\in[0,1]^X$, $v(x_0)>0$, such that
\begin{itemize}
\item[a)] $Mv \ge v$ 
\item[b)] for all $x$, $Mv(x)=v(x)$ if and only if   
$G(\mathbf 1 -(1-t) v;x)=1-(1-t)v(x), \forall t \in [0,1]$.
\end{itemize}
\item If there is global survival starting from $x_0$, then $\liminf_{n \in \N} \sqrt[n]{\sum_{x \in X} m^{(n)}_{x_0x}}\ge 1$.
\item If $X$ is a finite set then there is global survival starting from $x_0$ if and only if
\break $\liminf_{n \in \N} \sqrt[n]{\sum_{x \in X} m^{(n)}_{x_0x}}> 1$.
\end{enumerate}
\end{teo}

\begin{proof}
\begin{enumerate}
\item Fix $x_0 \in X$,
consider a path $\Pi:=\{x_0, x_1, \ldots, x_n=x_0\}$ and
consider its number of cycles $\#\{i=1,
\ldots,n:x_i=x_0\}$; the expected number of trails
based on such a path
is $\prod_{i=0}^{n-1} m_{x_i x_{i+1}}$.
This is also the expected number of particles
living at $x_0$,
descending from the original particle
at $x_0$ and whose genealogy is described by the path $\Pi$, that is, their mothers were at $x_{n-1}$, their
grandmothers at $x_{n-2}$ and so on.
We associate to the BRW starting at $x_0$
a Galton--Watson branching process with a different
time scale as in
\cite[Theorem 3.1]{cf:BZ} and \cite[Theorems 4.1 and 4.7]{cf:BZ2} where
the $n$-th generation is the set of particles living at $x_0$ whose
trail are based on paths with $n$ cycles.
This process is nonsingular due to Assumption~\ref{assump:1}, and its survival is equivalent to the local
survival of the BRW. The expected number of children in this branching process is $\Phi(x,x|1)$,
thus we have a.s.~local extinction if and only if $\Phi(x,x|1)\leq 1$, that is,
$\limsup_{n \in \N} \sqrt[n]{m^{(n)}_{xy}} \le 1$.
\item
Let $\bar q_n (x)$ and $\bar q(x)$ be the probability of global extinction before or at the $n$-th generation and the
probability of global extinction respectively, starting from a single
initial particle at $x$. Clearly $\bar q_{n+1}=G(\bar q_n)$ and $\bar q_n \to \bar q$ as $n \to \infty$.
If $\bar q(x_0)<1$ then take $z=\bar q$ as a solution of $G(z)\le z$ (remember that $G$ is continuous).
On the other hand, if there exists $z \in [0,1]^X$ such that
$z(x_0)<1$ and $G(z) \le z$, since $G$ is nondecreasing
and $\bar q_0=\mathbf 0 \le z$ then $\bar q_n \le z$ for all $n \in \N$,
hence $\bar q \le z$. Thus $\bar q(x_0)<1$.
\item
Let $z$ such that $G(z) \le z$, $z(x_0)<1$.
Define $v=\mathbf 1 - z$, take the derivative of the convex function $\phi(t):=G(\mathbf 1 -(1-t) v;x)-1+(1-t)v(x)$ at $t=1$ and
remember that  $\phi(0)\le \phi(1)=0$.
\item
If there is global survival starting from $x_0 \in X$ then there exists $v\in [0,1]^X$
such that $v(x_0)>0$ and $Mv \ge v$. Hence $M^nv \ge v$ for all $n \in \N$,
that is, $\sum_{y \in X} m^{(n
)}_{xy} v(y) \ge v(x)$; in particular, $\sum_{y \in X} m^{(n
)}_{x_0y} \ge v(x_0)>0$ and this implies
$\liminf_{n\to\infty} \sqrt[n]{\sum_y m^{(n)}_{x_0y}} \ge 1$.
\item
Since $X$ is finite there is global survival starting from $x_0$ if and only if there is
local survival starting from some $w \in X$ such that $x_0 \to w$, 
that is, 
$\limsup_{n \in \N} \sqrt[n]{m^{(n)}_{ww}}>1$.
Since $M$ is finite, it is easy to show that
\[
\liminf_{n \in \N} \sqrt[n]{\sum_{x \in X} m^{(n)}_{x_0x}}
= \max_{w \in X: x_0 \to w} \limsup_{n \in \N} \sqrt[n]{m^{(n)}_{ww}}
\]
whence there is global survival if and only if $\liminf_{n \in \N} \sqrt[n]{\sum_{x \in X} m^{(n)}_{x_0x}} >1$
(see also Assumption~\ref{assump:1} and Remark~\ref{rem:assump1}).
\end{enumerate}
\end{proof}

First of all, observe that it is easy to show, by using supermultiplicative arguments,
that  $\liminf_{n\to\infty} \sqrt[n]{\sum_y m^{(n)}_{xy}} \ge \limsup_{n\to\infty} \sqrt[n]{m^{(n)}_{xx}}$.
%
%
In terms of survival, studying a continuous-time BRW with
rates $\{\lambda k_{xy}\}_{x,y \in X}$ is equivalent  to studying
its discrete-time counterpart (that is, a BRW where $\{\mu_x\}_{x \in X}$ is
given by equations \eqref{eq:particular1} and \eqref{eq:counterpart}).
If we apply Theorem~\ref{th:equiv1} to a discrete-time counterpart then
we obtain \cite[Theorems 4.1, 4.3 and 4.7]{cf:BZ2}.
Moreover, according to \cite[Theorem 4.2]{cf:BZ2}(c), for the discrete-time counterpart
of a continuous-time BRW, global survival starting from
$x_0$ is equivalent to the existence of
$v\in[0,1]^X$, $v(x_0)>0$, such that
$Mv \ge v$. This is a necessary condition for global survival for
all discrete-time BRWs.
Finally, note that the proof of part (1) of the previous theorem
holds as well even
if $\{m_{xy}\}$ is unbounded or $m_{xy}=+\infty$ for some $x,y \in X$.
An independent
proof was given in \cite[Theorem 2.4]{cf:M08} using a different, more analytic, technique.

Speaking of global survival, it is easy to show that, given any solution of $G(z) \le z$, then $z(x)$
is an upper bound for the probability of extinction $\bar q(x)$. Moreover the existence of a solution
as in Theorem~\ref{th:equiv1}(2) is equivalent to the existence of a solution of $G(z)=z$
such that $z(x_0)<1$. From the proof we have that, if $\bar q$ is the possibly infinite-dimensional vector of
extinction probabilities, then $\bar q$ is the smallest solution of  $G(z)=z$ (and of $G(z) \le z$);
for this solution,
$z(x)<1$ simultaneously for all $x$ such that there is global solution starting from $x$.
Thus, if a BRW is irreducible and there is global survival
starting from one vertex then the solution $\bar q$ satisfies $\bar q(x)<1$ for all $x \in X$.
For a more detailed discussion on the generating function $G$ and its properties we refer to
\cite[Sections 2 and 3]{cf:BZ2}.
Of course, the possibility of computing an explicit solution of the inequality $G(z) \le z$
relies on the explicit knowledge of the generating function $G$. This can be done
in a few cases: for instance if equation~\eqref{eq:particular1} holds
(see Section~\ref{subsec:genfun}); in particular, for any discrete-time
counterpart of a continuous-time BRW, $G$ satisfies equation~\eqref{eq:Gcontinuous}.
We show how to manage this inequality in Examples~\ref{exm:noext} and \ref{exm:4.5}.
As for the inequalities in Theorem~\ref{th:equiv1}(1), (4) and (5), they can be easily
checked under some regularity conditions
(such as \textit{transitivity}, see Section~\ref{subsec:stronglocal} for the definition) or when
the BRW is an $\mathcal{F}$-BRW (see Definition~\ref{def:fBRW}).
In particular it has been shown in \cite{cf:GMPV09} that the inequality in Theorem~\ref{th:equiv1}(1)
can still be checked in particular random environments.


Besides, according to Theorem~\ref{th:equiv1}(1) the local survival depends only on $M$, hence if we have two
BRWs, say $(X,\mu)$ and $(X,\nu)$ with first-moment matrices $M$ and $\overline M$ respectively,
satisfying $m_{xy} \ge \overline m_{xy}$ (for all $x,y \in X$) then the local survival at $x_0 $ for
$(X,\nu)$ implies the local survival at $x_0$ for $(X,\mu)$. In Example~\ref{exm:noext}
we show that, for a general BRW,
the global survival does not depend only on $M$ nevertheless
a characterization of global survival in terms of $M$ alone holds for special classes of BRWs.
The first example is given by the class of discrete-time counterparts of continuous-time BRWs
and this is due to Theorem~\ref{th:equiv1}(2) and to the fact that, in this case, $G$ depends only
on $M$ (see equation~\eqref{eq:Gcontinuous}).
Another class is described by the following definition.

\begin{defn}\label{def:fBRW}
We say that a BRW $(X, \mu)$ is locally isomorphic to a BRW $(Y,\nu)$ if there exists a
surjective map $g:X\to Y$ such that
\begin{equation}
\label{eq:fgraph}
\nu_{g(x)}(f)=\mu_x\left(h:\forall y\in Y, f(y)=\sum_{z\in g^{-1}(y)}h(z)\right),
\quad \forall f\in S_Y.
\end{equation}
We say that $(X, \mu)$ is a $\mathcal F$-BRW if it is locally isomorphic to some
BRW $(Y,\nu)$ on a finite set $Y$.
\end{defn}

The idea behind the previous definition is that $g$ acts like a ``projection''
from $X$ onto $Y$ and, from the point of view of the BRW, all the vertices
in $g^{-1}(y)$ looks similar.
Besides, the class of $\mathcal F$-BRWs extends the class of multitype BRWs:
an $\mathcal F$-BRW identifies a multitype BRW (see the proof of
Theorem~\ref{th:fgraphf} for details) in a way that they have the same global behavior. Of course,
they may have different local behaviors.
We note that \textit{quasi-transitive} BRWs
(see Section~\ref{subsec:local} for the formal definition) are $\mathcal F$-BRWs.
Another example of an  $\mathcal F$-BRW is given by a BRW satisfying equation~\eqref{eq:particular1}
where $\rho_x$ is independent of $x \in X$, say $\rho_x=\rho$ for all $x \in X$; in this case one simply chooses $Y=\{0\}$,
that is,
a branching process with reproduction law $\rho$.
There are $\mathcal F$-BRWs which are not quasi-transitive: an
example is the discrete-time counterpart of the continuous-time BRW
given in \cite[Example 3.1]{cf:BZ}.

The function $g$ induces a map $\pi_g: S_X \to S_Y$ defined as
$\pi_g(f)(y)=\sum_{x \in g^{-1}(y)} f(x)$ hence equation~\eqref{eq:fgraph}
becomes  $\nu_{g(x)}(\cdot)= \mu_{x}(\pi_g^{-1}(\cdot))$.
Clearly if $\{\eta_n\}_{n\in \N}$ is a realization of $(X,\mu)$ then
$\{\pi_g(\eta_n)\}_{n\in \N}$ is a realization of $(Y,\nu)$.
Moreover it is easy to show that, for all $x \in X, \, y \in Y$,
$\widetilde m_{g(x) y}:=\sum_{w \in S_Y} w(g(x)) \nu_{y}(w)=
\sum_{f \in S_X} \sum_{z \in g^{-1}(y)} f(z) \mu_x(f)=
\sum_{z \in g^{-1}(y)} m_{xz}$ 
(that is, $\widetilde m_{g(x)\, y}= \pi_g(m_{x\, \cdot})(y)$).
This means that the expected number of offsprings
at $y$ of a particle living at $g(x)$ (on the projected BRW $(Y,\nu)$)
is the sum of the expected numbers of offsprings at $z$
of a particle living at $x$ (on the BRW $(X,\mu)$) over all $z \in X$ whose
projection is $y$.
Thus $
\sum_{y \in Y} \widetilde m_{g(x)\, y}
=\sum_{z \in X} m_{xz}
$.
By induction on $n \in \N$ one can prove that, for all $x \in X, \, y \in Y, \, n \in \N$, 
we have $\widetilde m^{(n)}_{g(x) y}=\sum_{z \in g^{-1}(y)} m^{(n)}_{xz}$ whence
$\sum_{y \in Y}\widetilde m^{(n)}_{g(x) y}=\sum_{z \in X} m^{(n)}_{xz}$.
It is not difficult to show that $(X,\mu)$ is locally isomorphic to $(Y,\nu)$
if and only if
$G_X(z \circ g|x)=G_Y(z|g(x))$ for all $z \in [0,1]^Y$, $x \in X$.

The following result characterizes the global survival of a
$\mathcal F$-BRW in terms of $M$.

\begin{teo}\label{th:fgraphf}
Let $(X,\mu)$ is locally isomorphic to $(Y,\nu)$ and
consider the following:
\begin{enumerate}
\item there is global survival for $(X,\mu)$ starting from $x_0 \in X$,
\item there is global survival for $(Y,\nu)$ starting from $g(x_0) \in Y$,
\item $\liminf_{n\to\infty} \sqrt[n]{\sum_y m^{(n)}_{x_0y}}>1$;
\end{enumerate}
then $(1) \Longleftrightarrow (2)$. Moreover if $Y$ is finite
(hence $X$ is an $\mathcal F$-BRW) then $(3) \Longleftrightarrow (2)$.
\end{teo}

\begin{proof}
\indent $(1) \Longleftrightarrow (2)$.
I
It is easy to show, by induction on $n \in \N$, that
$G^{(n)}_X(z \circ g|x)=G^{(n)}_Y(z|g(x))$ for all $z \in [0,1]^Y$, $x \in X$ (where $G^{(n+1)}(z)=G^{(n)}(G(z))$).
This implies that $\bar q^Y_{n}(g(x))=\bar q^X_{n}(x)$ for all $n \in \N$. Hence,
as $n \to \infty$, we have that the minimal fixed points
of these generating functions satisfy $\bar q^Y(g(x))=\bar q^X(x)$.\\
\indent $(1) \Longleftrightarrow (3)$.
Since, 
for all $n \in \N$, we have
$\sum_{z \in X} m^{(n)}_{xy}=\sum_{y \in Y} \widetilde m^{(n)}_{g(x)y}$
then
\[
\liminf_{n \in \N} \sqrt[n]{\sum_{x \in X} m^{(n)}_{x_0x}} = \liminf_{n \in \N} \sqrt[n]{\sum_{y \in Y} \widetilde m^{(n)}_{g(x_0)y}}.
\]
The claim follows from Theorem~\ref{th:equiv1}(5) being $Y$ finite.
\end{proof}

Since, within some classes, the global behavior can be characterized completely by $M$, one can wonder if the same holds for a general BRW
 or, alternatively, if given two generic BRWs with the same
first-moment matrix then they have the same global behavior. In particular
one could conjecture that at least one of the two necessary conditions
given in Theorem~\ref{th:equiv1}(3)~and~(4) is also sufficient.
All these conjectures are false as the following example shows
(the main tool is Theorem~\ref{th:equiv1}(2)).

\begin{exmp}\label{exm:noext}

Let $X=\mathbb N$ and consider the family of BRWs $(\mathbb N,\mu)$ with
$\mu_i=p_i \delta_{n_i \ident_{\{i+1\}}} +(1-p_i)\delta_{\bf 0}$
(where $\ident_{\{i+1\}} \in S_{\mathbb N}$ is defined by
$\ident_{\{i+1\}}(x)=1$ if $x=i+1$ and $0$ otherwise).
Roughly speaking, each particle at $i$ has $n_i$ children at $i+1$ with probability $p_i$ and
no children at all with probability $1-p_i$.
According to Theorem~\ref{th:equiv1}(2) global survival starting from $0$ is equivalent
to the existence of $z\in [0,1]^{\mathbb N}$, $z(0)<1$,
such that $G(z|i)\le z(i)$, for all $i$ where
$G(z|i)=p_i z(i+1)^{n_{i}} +1-p_i$. Note that $pz^n+1-p \to 1$
if $p\to 0$ or $z \to 1$.

Clearly if $n_i=n$, $p_i=p$ and $np>1$ the BRW survives globally
(take for instance $n=4$ and $p=1/2$).
Let us suppose that $p_i=2/n_i$. We construct iteratively
a sequence $\{n_i\}_{i \in \mathbb N}$ such that the unique solution
of $G(z)\le z$ is $z(i)=1$ for all $i \in \mathbb N$.

Clearly $G(z)\le z$ implies
\begin{equation}\label{eq:partsyst}
\begin{cases}
z(0) \ge \frac{2}{n_0} z(1)^{n_1} + 1-\frac{2}{n_0}\\
z(1) \ge \frac{2}{n_1} z(2)^{n_2} + 1-\frac{2}{n_1}\\
\ldots\\
z(k) \ge \frac{2}{n_k} z(k+1)^{n_k} + 1-\frac{2}{n_k}\\
z(k+1) \ge 1-\frac{2}{n_{k+1}}.\\
\end{cases}
\end{equation}
for all $k \in \mathbb N$. Let $n_0=4$ and suppose
we already fixed $\{n_i\}_{i=0}^k$. If $n_{k+1} \to \infty$ then
a solution of equation~\eqref{eq:partsyst} satisfies
$z(i) \to 1$ for all $i \le k+1$. Choose $n_{k+1}$ such that
$z(i) \ge k/(k+1)$ for all $i \le k+1$. This implies
that the unique solution of the family of systems (depending on $k$)
given by equation~\eqref{eq:partsyst} is $z(i)=1$ for all $i \in \mathbb N$.
Thus this is the only solution of $G(z)\le z$ and the BRW does not
survive globally a.s.
In this case $\liminf_{n\to\infty} \sqrt[n]{\sum_y m^{(n)}_{x_0y}}=2$, thus
this example shows in particular that $\liminf_{n\to\infty} \sqrt[n]{\sum_y m^{(n)}_{x_0y}}>1$
does not imply, in general, global survival.

The first-moment matrix of the BRW above is
not  irreducible and the BRW can be identified with a time-inhomogeneous branching
process; a slight modification allows us to construct an irreducible BRW.
We just sketch the main steps.

Again let $X=\mathbb N$ and consider the family of BRWs
$\mu_i=p_i \delta_{n_i \ident_{\{i+1\}}+\ident_{\{i-1\}}} +(1-p_i)\delta_{\bf 0}$ (for all $i\ge 1$) and
$\mu_0=p_0 \delta_{n_0 \ident_{\{1\}}} +(1-p_0)\delta_{\bf 0}$.
In this case each particle at $i \ge 1$ has $n_i$ children at $i+1$ and $1$ at $i-1$ with probability $p_i$ and
no children at all with probability $1-p_i$; each particle at $0$ has the same behavior as in the previous example.
The generating function $G$ is
\[
G(z|i)=
 \begin{cases}
p_i z(i+1)^{n_{i}}z(i-1) +1-p_i & i \ge 1\\
p_0 z(1)^{n_{0}} +1-p_0 & i =0.
 \end{cases}
\]
$G(z)\le z$ implies, for all $k$,
\begin{equation}\label{eq:partsyst2}
\begin{cases}
z(0) \ge p_0 z(1)^{n_0} + 1-p_0\\
z(1) \ge p_1 z(2)^{n_1}z(0) + 1-p_1\\
\ldots\\
z(k) \ge p_k z(k+1)^{n_k}z(k-1) + 1-p_k\\
z(k+1) \ge 1-p_{k+1}.\\
\end{cases}
\end{equation}
It is not difficult to prove that, if  $p_{k+1} \to 0$ then $z(k+1)\to 1$) and
the set of solutions of equation~\eqref{eq:partsyst2} is eventually contained in any $\varepsilon$-enlargements of the set of vectors
$(z_0(1), z_0(2), \ldots, z_0(k),1)$, where
$(z_0(1), z_0(2), \ldots, z_0(k))$ is ranging in the set of solutions of
\begin{equation}\label{eq:partsyst3}
\begin{cases}
z(0) \ge p_0 z(1)^{n_0} + 1-p_0\\
z(1) \ge p_1 z(2)^{n_1}z(0) + 1-p_1\\
\ldots\\
z(k-1) \ge p_{k-1} z(k)^{n_{k-1}}z(k-2) + 1-p_{k-1}\\
z(k) \ge p_k z(k-1) + 1-p_{k}.\\
\end{cases}
\end{equation}
Let us study this last equation. We note that if $n_i p_i p_{i+1} \le (1-\eps)/2$ for all $i \in \N$ and for some $\eps>0$ then
there is a unique solution of  equation~\eqref{eq:partsyst3}, that is $z(i)=1$ for all $i=0, \ldots, k$.
Indeed  equation~\eqref{eq:partsyst3} represents the system $\widetilde G(z) \le z$ for an irreducible BRW on $\{0,1, \ldots, k\}$
where
\[
 \widetilde \mu_i=
\begin{cases}
  p_0 \delta_{n_0 \ident_{\{1\}}} +(1-p_0)\delta_{\bf 0} & \text{if } i=0\\
  p_i \delta_{n_i \ident_{\{i+1\}}+\ident_{\{i-1\}}} +(1-p_i)\delta_{\bf 0} & \text{if } i=1, \ldots, k-1\\
  p_k \delta_{\ident_{\{k-1\}}} +(1-p_k)\delta_{\bf 0} & \text{if } i=k.\\
 \end{cases}
\]
Indeed, since the graph is finite and connected,
according to Theorem~\ref{th:equiv1}(2)~and~(5) there exists a solution $z \not = \mathbf{1}$ of $\widetilde G(z) \le z$
if and only if $\liminf_{n\to\infty} \sqrt[n]{\sum_j \widetilde m^{(n)}_{ij}}>1$ for some ($\Longleftrightarrow$~for all) $i \in \{0,1, \ldots, k\}$;
but, again since the graph is finite, the previous conditions are equivalent to
$\limsup_{n\to\infty} \sqrt[n]{\widetilde m^{(n)}_{ii}}>1$ for some ($\Longleftrightarrow$~for all)
$i \in \{0,1, \ldots, k\}$. Elementary computations show that
\[
\widetilde m^{(n)}_{ii} \le
\begin{cases}
 \frac1{n+1}\binom{n+1}{n/2} \big (\frac{1-\eps}{2} \big )^{n} \le
\binom{n}{n/2} \big (\frac{1-\eps}{2} \big )^{n}
& \text{if $n$ is even} \\
0 & \text{if $n$ is odd}
\end{cases}
\]
(remember that $\widetilde m_{i \, i+1} \widetilde m_{i+1 \, i}=p_{i}n_{i}p_{i+1}<(1-\eps)/2$)
which implies $\limsup_{n\to\infty} \sqrt[n]{\widetilde m^{(n)}_{ii}}\le 1-\eps$.
This proves that the unique solution of equation~\eqref{eq:partsyst3} is $z(i)=1$ for all $i=0, \ldots, k$.

As before, the trick to prove our goal is to choose the sequences $\{p_i\}_{i\in \N}$ and $\{n_i\}_{i\in \N}$ such that
$p_i \to 0$ fast enough and $p_i n_i=2$ for all $i \in \N$.
Note that if $p_i=2/n_i <(1-\eps)/4$ for all $i \in \N$ then $p_{i+1}p_{i}n_{i}<(1-\eps)/2$.

If $k=1$ then we can choose $n_1$ such that $z(i)> 1/2$ for all $i \le 1$. Indeed if $n_1 \to \infty$ then
$p_1 \to 0$ and both $z(1), z(0) \to 1$.


Suppose we fixed $n_0, \ldots, n_k$ such that any solution of equation~\eqref{eq:partsyst2} satisfies
$z(i) \ge k/(k+1)$ for all $i \le k$ and such that $p_i <(1-\eps)/4$ for all $i=0, \ldots, k$.
%
If $n_{k+1} \to \infty$ then $z(k+1) \to 1$ hence any solution of equation~\eqref{eq:partsyst2}
must converge as before to a solution of equation~\eqref{eq:partsyst3}.
Hence $n_{k+1} \to \infty$ implies $z(i) \to 1$ for all $i \le k+1$ and we can choose
$n_{k+1}$ such that $z(i) \ge (k+1)/(k+2)$ for all $i \le k+1$. This yields the conclusion.

Finally we note that if the BRW is given by
$\mu_i=1/2 \, \delta_{4 \ident_{\{i+1\}}}+p_i \delta_{\ident_{\{i-1\}}} +(1/2-p_i)\delta_{\bf 0}$ (for all $i\ge 1$) and
$\mu_0=1/2 \delta_{4 \ident_{\{1\}}} +1/2 \delta_{\bf 0}$ (where $p_i$ is the same as before) then
it survives globally, hence, even for irreducible BRWs, global survival does not depend only
on the first-moment matrix $M$ and
 $\liminf_{n\to\infty} \sqrt[n]{\sum_y m^{(n)}_{x_0y}}>1$
does not imply, in general, global survival.
\end{exmp}

Another possible question arises from Theorem~\ref{th:equiv1}: is it true that $\sum_{y \in X} m_{xy}<1$
for all $x \in X$ implies global extinction? According to the following example
(see also \cite[Example 1]{cf:BZ2}), the answer is negative.

\begin{exmp}\label{exm:4.5}
As before, we start by giving an example which is not irreducible, later on we modify the
process in order to obtain an irreducible BRW.

Let $X=\N$, $\{p_n\}_{n\in\N}$ be a sequence in $(0,1]$ and suppose that a particle at $n$
has one child at 
$n+1$ with
probability $p_n$ and no children with probability $1-p_n$
(this is the reducible process of the previous example with $n_i=1$
for all $i \in \N$).
The generating function of this process is
$\widetilde G(z|n)=1-p_n + p_n z(n+1) $. 
Again this BRW can be identified with a time-inhomogeneous branching process which
has a probability of extinction, starting with one particle at $n$, equal to
$z(n)=1-\prod_{i=n}^{\infty} p_i$; hence it survives wpp,
if and only if $\sum_{i=1}^\infty (1-p_i) < +\infty$.
It is straightforward to check that $z$ is a solution of $G(z)=z$.

This process is stochastically dominated by the irreducible BRW where
each particle at  $n\ge1$ has one child at $n+1$ with
probability $p_n$, one child at $n-1$ with probability $(1-p_n)/2$
(if $n=0$ then it has one child of type $0$ with probability $(1-p_0)/2$)
and no children at all with probability $(1-p_n)/2$.
The generating function $G$ can be explicitly computed
\[
G(z|n)=
\begin{cases}
\frac{1-p_n}{2} + \frac{1-p_n}{2} z(n-1)+ p_n z(n+1) & n \ge 1 \\
\frac{1-p_0}{2} + \frac{1-p_0}{2} z(0) + p_0 z(1) & n=0.\\
\end{cases}
\]

By coupling this process with the previous one or, simply, by applying Theorem~\ref{th:equiv1}(2)
($z(n)=1-\prod_{i=n}^{\infty} p_i$ is a solution of $G(z) \le z$) one can prove that
$\sum_{i=1}^\infty (1-p_i) < +\infty$ implies global survival. Note that here
$\sum_{j \in \N} m_{ij} = (1+p_i)/2 <1$; clearly, $\liminf_{n \to \infty} \sqrt[n]{\sum_{j \in \N} m^{(n)}_{ij}}=1$.
\end{exmp}

Analogous examples could be constructed for continuous time BRWs as well.
For instance, an example of a continuous-time BRW which survives globally at the global
critical point $\lambda=\lambda_w$ can be found in \cite[Example 3]{cf:BZ2}.

\subsection{Local survival and strong local survival}
\label{subsec:stronglocal}

We know that $\bar q$ is the smallest fixed point of $G(z)$ in $[0,1]^X$,
where $\bar q(x)$ is the probability of global extinction
starting from $x$. Moreover, if a sequence
$\{z_n\}_{n \in \mathbb N}$, defined recursively by
$z_{n+1}=G(z_n)$, has a limit $z$,  then $z=G(z)$ and, if the
sequence is non decreasing, then $z = \bar q$ if and only if $z_0
\le \bar q$.

Define $q_n(x,A)$ as the probability of extinction before
generation $n+1$ in $A$ starting with one particle at $x$. It is
clear that $\{q_n(x,A)\}_{n \in \N}$ is a nondecreasing sequence satisfying
\[
\begin{cases}
 q_{n}(\cdot,A)=G(q_{n-1}(\cdot, A)),& \quad \forall n \ge 1\\
 q_0(x,A)=0, &\quad \forall x \in A,
\end{cases}
\]
hence there is a limit $q(x,A)=\lim_{n \to \infty} q_n(x, A) \in
[0,1]^X$ which is the probability of local extinction in $A$
starting with one particle at $x$. Clearly $q(\cdot,A)=G(q(\cdot,
A))$, hence these probabilities are
fixed points of $G$. It is easy to show that $A \subseteq B$ implies $q(\cdot,A)
\ge q(\cdot,B)$; in particular $q(\cdot, \emptyset)= \mathbf 1$,
and $q(\cdot, X)=\bar q$. The probability $q(\cdot, \{y\})$ of local extinction at $y$
starting from $x$ is denoted simply by $q(\cdot, y)$.
In general $q(\cdot,y)
\ge \bar q$ for all $y \in X$ and $q(\cdot,y)= \bar q$
if and only if $q_0(\cdot,y) \le \bar q$. 
If for some $y \in X$ we have $q(\cdot,y)= \bar q$ then the global survival starting from $x$ implies
the strong local survival at $y$ starting from $x$.

In the special case when the set $\mathcal A$ of fixed points for $G$ is reduced to the minimal one
$\{\bar q, \mathbf 1\}$, 
then either there is no local survival at $y$ starting from
every fixed vertex $x$ (that is, $q(\cdot,y)=\mathbf 1$)
or the global survival at $x$ implies the strong local survival
at $y$ (that is, $q(\cdot,y)=\bar q$). Clearly, since $\bar q$ is the smallest fixed point
of $G$ then $\mathcal A \subseteq \prod_{x \in X} [\bar q(x),1]$.
Thus, it would be important to find conditions for either $\mathcal A=\{\bar q, \mathbf 1\}$
or, when it is possible, for $q(\cdot,y)= \bar q$ for all $y \in X$.
%
Discussing such conditions in details goes beyond the purpose of this paper, but we want to give one
result about the strong local survival of a class of regular discrete-time BRWs to give an idea
of what can be done.

\begin{defn}\label{def:invariantmu}
Let $\gamma:X \to X$ be an injective map. 
We say that $\mu=\{\mu_x\}_{x \in X}$ is $\gamma$-invariant if
for all $x,y \in X$ and $f \in S_X$ we have
$\mu_x(f)=\mu_{\gamma(x)}(f \circ \gamma^{-1})$.
Moreover $(X,\mu)$ is \textit{transitive} if and only if there exists
$x_0 \in X$ such that for all $x \in X$ there exists a bijective
map $\gamma:X\to X$ satisfying $\gamma(x_0)=x$ and $\mu$ is $\gamma$-invariant.
\end{defn}

\begin{pro}
Let $(X, \mu)$ be a discrete-time  BRW
with an irreducible matrix $M$. Then there is local survival at some $y \in X$ starting from some $x\in X$
if and only if there is local survival at all $y \in X$ starting from every fixed $x \in X$. Moreover, if the BRW is transitive
then local survival
implies strong local survival.
\end{pro}

\begin{proof}
Since the BRW is transitive, then the probability $\bar q(x)$ of global extinction
does not depend on the starting point $x$
and, analogously, the probability $q(x,x)$ of local extinction at $x$ starting from $x$ is independent of $x \in X$.
Moreover the matrix $M$ is irreducible then 
for all $x,y \in X$ we have $q(x,x)=q(x,y)$. 
This,
along with transitivity, implies
that there are $\alpha, \beta \in [0,1]$ such that $\alpha \le \beta$ and $\bar q = \alpha \mathbf 1$ and
$q(\cdot, y)= \beta \mathbf 1$ for all
$y \in X$. Note that $h(t):= G( t \mathbf 1;x)=\sum_{n \in \N} \rho_x(n) t^n$
does not depend on $x$; moreover
$h(\alpha)=\alpha$, $h(\beta)=\beta$ and $h(1)=1$.
If there is local survival then
$\beta<1$ and this, along with Assumption~\ref{assump:1}, implies that $h$ is strictly convex,
thus
$\alpha=\beta$.
This implies $q(\cdot, y)=\bar q$ for all $y \in X$.
\end{proof}
\noindent For a similar result concerning the strong survival of a BRW on a random environment on $\Z$ see~\cite{cf:GMPV09}.

\begin{rem}\label{rem:regularity}
We note that if a continuous-time BRW is $\gamma$-invariant according
to \cite[Section 5]{cf:BZ3} then the discrete-time counterpart
is $\gamma$-invariant.
Hence, the previous result can be helpful to study the regularity of the extinction probabilities $q_\lambda(x,y)$ and
$\bar q_\lambda$, as functions of $\lambda$, of a discrete-time counterpart of a continuous-time BRW
(see equation~\eqref{eq:counterpart}). They are both nonincreasing functions;
for instance if $k(x)=\sum_{y \in X} k_{xy}=\bar k$ 
then
$\bar q_\lambda(x)=\min (1/\lambda \bar k,1)$ for all $x \in X$.

Remember the explicit analytic form of $G$ given by equation~\eqref{eq:Gcontinuous}:
since $M=\lambda K$ there is a dependence of $G=G_\lambda$ on $\lambda$; it is straightforward to show that
\[
\sup_{x \in X} |G_\lambda(z;x)-G_{\lambda^\prime}(z;x)| \le |\lambda-\lambda^\prime| \sup_{x \in X}\sum_{y \in X} k_{xy}.
\]
Moreover $\lambda \ge \lambda^\prime$ implies $G_\lambda(z)\le G_{\lambda^\prime}(z)$ for all $z \in [0,1]^X$.
Using these properties, along with the fact that $\bar q_\lambda$ is the smallest solution of $G_\lambda(z) \le z$
(see \cite{cf:BZ2} for details) it is easy to prove that $\bar q_\lambda \uparrow \bar q_{\lambda^\prime}$ as
$\lambda \downarrow \lambda^\prime$.
In the irreducible case, $\bar q_\lambda= \mathbf 1$ for all $\lambda < \lambda_w$ and \cite[Example 3]{cf:BZ2}
shows that it may happen that $\bar q_{\lambda_w} < \mathbf 1$ which implies that $\bar q_\lambda$ can be discontinuous from
the left.

As for $q_\lambda(x,y)$, clearly in the irreducible case $q_\lambda(\cdot,y)=\mathbf 1$ for all $\lambda \le \lambda_s$.
On the other hand in the irreducible and transitive case we just proved that $\bar q_\lambda = q_\lambda(\cdot, y)$ for
all $\lambda > \lambda_s$ hence if $\lambda_s > \lambda_w$ in general $\bar q_{\lambda_s}< \mathbf 1$, thus
$\lim_{\lambda \downarrow \lambda_s} q_\lambda(\cdot, y)= \bar q_{\lambda_s} < \mathbf 1 =  q_{\lambda_s}(\cdot, y)$.
Hence
$q_\lambda(\cdot, y)$ can be discontinuous from the right.
\end{rem}

\section{Spatial approximation}
\label{sec:spatial}

\subsection{Generalization of a Theorem of Sarymshakov--Seneta}\label{susec:seneta}

Given a matrix $M=(m_{xy})_{x,y \in X}$, 
recall the usual classification of indices of a matrix as described in \cite[Chapter 1]{cf:Sen}.
For any index $x$ we denote by $[x]$ its \textit{class}, that is, the set of indices
which communicate with $x$.
We define the convergence parameters $R(x,y):=\left (\limsup_{n \in \N} \sqrt[n]{m^{(n)}_{xy}} \right )^{-1}$
and $R:= \inf_{x,y \in X} R(x,y)$;
it is well known that $\limsup_{n \in \N} \sqrt[n]{m^{(n)}_{xy}}=\limsup_{n \in \N} \sqrt[n]{m^{(n)}_{x_1y_1}}$ if
$[x]=[x_1]$ and $[y]=[y_1]$; 
in particular $R(x,y)$ is independent of $x,y$ if the matrix is irreducible.

Let $\{X_n\}_{n \in \N}$  be a 
sequence of subsets of $X$ 
and denote by $_nR$ the convergence parameter of $M_n=(m_{xy})_{x,y \in X_n}$;
clearly, if the sequence $\{X_n\}_{n \in \N}$ is nondecreasing, we have that $_nR \geq {_{n+1}R}$.
The following theorem  generalizes \cite[Theorem 6.8]{cf:Sen}
(note that the submatrices $\{M_n\}_{n \in \N}$ are not necessarily irreducible); it
is the key to prove our main
result about spatial approximation (Theorem~\ref{th:spatial}).

\begin{teo}\label{th:genseneta}
 Let $\{X_n\}_{n \in \N}$  be a general sequence of subsets of $X$ such that
$\liminf_{n \to \infty} X_n =X$. Then for all $x_0 \in X$ we have $_nR(x_0,x_0) \rightarrow R(x_0,x_0)$.
Moreover if $M$ is irreducible and $M_n=(m_{xy})_{x,y \in X_n}$ then $_nR \rightarrow R$ as $n \to \infty$ and, in particular,
for all $x_0 \in X$ we have $_nR(x_0,x_0) \rightarrow R$.
\end{teo}

\begin{proof}
Suppose that $M$ is irreducible.
We start by showing that if $\{X_n\}_{n \in \N}$  is a nondecreasing
sequence of subsets of $\N$ such that $\bigcup_{n \in \N} X_n = X$
then $_nR \downarrow R$ as $n \to \infty$ and $_nR(x_0,x_0) \rightarrow R$
for all $x_0 \in X$.
If $M_n$ are all irreducible then the claim follows easily from \cite[Theorem 6.8]{cf:Sen}.
In the general case, fix an index $x_0 \in X$ and consider the sequence of sets
$\{J_n\}_{n \in \N}$ where $J_n$ is the class of $x_0$ in $X_n$.
Given any index $x$, we have that $x \in J_n$ eventually as $n\to\infty$; indeed
if $A$ is the set of vertices in a path connecting $x_0$ to $x$ and back
(which exists since $M$ is irreducible) then eventually $A \subseteq X_n$ which
implies $A \subseteq J_n$ thus $\bigcup_n J_n=X$. Let us call $_n\widetilde R$ the convergence parameter
of $\widetilde M_n=(m_{xy})_{x,y \in J_n}$. Since $\widetilde M_n$ is an irreducible
submatrix of $M_n$
then, according to \cite[Theorem 6.8]{cf:Sen}, $_n\widetilde R \downarrow R$. On the other hand
$R \leq {_nR} \leq {_nR(x_0,x_0)} \leq {_n\widetilde R(x_0,x_0)} = {_n\widetilde R}$ which yields the conclusion
for a nondecreasing sequence of subsets.

If $\{X_n\}_{n \in \N}$ is not monotone, then consider a nondecreasing sequence
$\{X^\prime_n\}_{n \in \N}$ of finite subsets of $X$ such that $\bigcup_{n \in \N} X^\prime_n = X$.
For any $n$ there is $r_n$ such that for all $r \ge r_n$ we have $X_{r} \supseteq X^\prime_n$.
Clearly, for all $r \ge r_n$,
\[
R \leq {_rR} \leq {_{r}R(x_0,x_0)} \leq {_nR^\prime(x_0,x_0)} \downarrow R
\]
as $n \to \infty$ (where $_nR^\prime$ is
the convergence parameter of $M^\prime_n=(m_{xy})_{x,y \in X^\prime_n}$).

Finally, if $M$ is not irreducible, consider $\widetilde{M}=(m_{xy})_{x,y \in [x_0]}$
and $\widetilde{M}_n=(m_{xy})_{x,y \in [x_0] \cap X_n}$. It is easy to show that
$\widetilde{R}(x_0, x_0) = R(x_0,x_0)$ and $_n\widetilde{R}(x_0, x_0)={_nR(x_0,x_0)}$
for all $n \in \N$. Since $\widetilde M$ is irreducible, the first part of the proof yields the conclusion.
\end{proof}

Note that in the previous theorem the subsets $\{X_n\}_{n \in \N}$ can be chosen
arbitrarily; in particular they may be finite proper subsets.

\subsection{Application to BRWs}\label{subsec:spatialBRW}



Given a sequence of BRWs $\{(X_n,\mu_n)\}_{n \in \N}$ such that $\liminf_{n \to \infty} X_n=X$,
we define $m(n)_{xy}:=\sum_{f \in S_{X_n}} f(y) \mu_{n,x}(f)$ and
the corresponding sequence of submatrices $\{M_n\}_{n \in \N}$ of $M$ where $M_n=(m(n)_{xy})_{x,y \in X_n}$.
The main goal of this section is to 
investigate if the survival  of $(X,\mu)$ can guarantee
the survival of $(X_n,\mu_n)$ for all sufficiently large $n$.
The following theorem is the main result of this section; note that
for all $x,y \in X$, $m(n)_{xy}$ is well defined eventually as $n \to \infty$.
In this result we are not assuming that the BRW is irreducible.

\begin{teo}\label{th:spatial}
Let us 
fix a vertex $x_0 \in X$. If $\liminf_{n \to \infty} X_n=X$ and
$m(n)_{xy} \le m_{xy}$ for all $x,y \in X_n$, $n \in \N$ and
 $m(n)_{xy} \to m_{xy}$ as $n \to \infty$ then
\begin{enumerate}
\item
$(X,\mu)$ dies out locally (resp.~globally) a.s.~starting from $x_0$ $\Longrightarrow$ $(X_n,\mu_n)$ dies out locally (resp.~globally)
a.s~starting from $x_0$  for all $n \in \N$;
\item
$(X,\mu)$ survives locally wpp starting from $x_0$  $\Longrightarrow$ $(X_n,\mu_n)$ survives locally wpp starting from $x_0$ eventually as $n \to\infty$.
\end{enumerate}
\end{teo}

\begin{proof}
\begin{enumerate}
\item
It follows by coupling the BRW$(X_n,\mu_n)$ with the subcritical BRW$(X,\mu)$
as described in Section~\ref{subsec:1}.
\item
Let us fix a sequence $\{Y_n\}_{n \in \N}$ of finite subsets of $X$ such that $\liminf_{n \to \infty} Y_n =X$.
By Theorem~\ref{th:equiv1}(1) there exists $\eps>0$ such that $\limsup_{i\to\infty}\sqrt[i]{m^{(i)}_{x_0x_0}}>1+\eps$.
Consider the sequence of submatrices $A_n=(a(n)_{xy})_{x,y \in Y_n}$ where $a(n)_{xy}:=m_{xy}/(1+\eps)$.
Using Theorem~\ref{th:genseneta} 
we have that
\[
 \lim_{n \to \infty} \limsup_{i\to\infty}\sqrt[i]{a(n)^{(i)}_{x_0x_0}} = \limsup_{i\to\infty}\sqrt[i]{m^{(i)}_{x_0x_0}}/(1+\eps) >1,
\]
as $n \to \infty$. Let $\bar n$ such that $\limsup_{i\to\infty}\sqrt[i]{a(\bar n)^{(i)}_{x_0x_0}}>1$.
Moreover since $Y_{\bar n}$ is finite
there exists $n_0$ such that for all $n \ge n_0$ we have $m(n)_{xy} \ge m_{xy}/(1+\eps)=a(\bar n)_{xy}$ for all $x,y \in Y_{\bar n}$, thus
\[
 \limsup_{i\to\infty}\sqrt[i]{m(n)^{(i)}_{x_0x_0}}\ge \limsup_{i\to\infty}\sqrt[i]{a(\bar n)^{(i)}_{x_0x_0}}>1
\]
for all $n \geq n_0$. Theorem~\ref{th:equiv1}(1) yields the conclusion.
\end{enumerate}
\end{proof}

Note that in the language of continuous-time BRWs
(see \cite{cf:BZ} and \cite{cf:BZ2} for details), the claim of the previous theorem
is $\lambda_s((X_n, \mu_n),x_0) \to \lambda_s((X,\mu),x_0)$; hence it is a generalization of \cite[Theorem 3.1]{cf:BZ3}.

Among all the possible choices of the sequence
$\{(X_n,\mu_n)\}_{n \in \N}$ there is one which is
\textit{induced} by $(X,\mu)$ on the subsets $\{X_n\}_{n\in \N}$;
more precisely, one can take $\mu_n(g):=\sum_{f \in S_X:f|_{X_n}=g} \mu_x(f)$
for all $x \in X_n$ and $g \in S_{X_n}$.
Roughly speaking, this choice means that all the reproductions outside $X_n$ are suppressed.
Note that, in this case,
$m(n)_{xy}=m_{xy}$ for all $x,y \in X_n$; the result in this particular case is implicitly
used, for instance, in the proof of \cite[Theorem 2.4]{cf:GMPV09}.

\begin{rem}
Theorem~\ref{th:spatial} deals mainly with local survival. One can wonder what
can be said about global survival. Clearly if the $(X,\mu)$ process survives 
globally and locally then eventually $(X_n,\mu_n)$ survives locally and thus globally.

The question is nontrivial when $(X,\mu)$ survives globally but not locally, which we assume henceforth in this remark.

In this last case,
if $X_n$ is finite for every $n \in \N$ and the graph $(X_n, E_{\mu_n})$ is connected
then there is no distinction between global and local survival for the process $(X_n,\mu_n)$; in
particular $(X_n,\mu_n)$ dies out (locally and globally) a.s.~for all values of $n \in \N$.

On the other hand, the case where $X_n$ is finite for every $n \in \N$ and the graph $(X_n, E_{\mu_n})$ is not connected
is more complicated and can be treated as in \cite[Remark 4.4]{cf:BZ2}.

When $X_n$ is infinite for infinitely many values of $n$,
one cannot always expect to have global survival for sufficiently large values of $n$.
For continuous-time BRWs the counterexample can be constructed using \cite[Remark 3.10]{cf:BZ},
hence a discrete-time counterexample would be the discrete-time counterpart.
Nevertheless a discrete-time counterexample can be constructed directly as follows
(we use the notation introduced in Section~\ref{subsec:discrete}).

Fix an infinite, transitive and connected graph $X$ and consider a reproduction law as
in equation~\eqref{eq:particular1} such that the random walk $\{p(x,y)\}_{x,y \in X}$ has a spectral radius
$\theta <1$. Suppose that $\bar \rho_x=\bar \rho \in (1, \theta^{-1})$ is independent of $x$.
Hence (see equation~\eqref{eq:meanparticular}) according to Theorem~\ref{th:fgraphf} there is global survival
since $\sqrt[n]{\sum_{y \in X} m^{(n)}_{xy}}= \bar \rho >1$.
On the other hand there is no local survival since, according to Theorem~\ref{th:equiv1}(1),
$\limsup_{n \to \infty}  \sqrt[n]{ m^{(n)}_{xx}} = \bar \rho \theta <1$.

Fix a vertex $o \in X$, an infinite self-avoiding path $\gamma=\{x_i\}$ in $X$ starting from $o$ and
define
$X_n:=B(o,n) \cup \gamma$ (where $B(o,n)$ denotes the ball with center $o$ and radius $n$). Finally,
consider the restricted laws $\mu_n$ of $\mu$ to $X_n$ as described in Section~\ref{subsec:1}.

Adapting the arguments of \cite[Remark 3.10]{cf:BZ}, since there is no local survival
in $X_n$ and the ball $B(o,n)$ is finite, one can prove that there is global survival on $X_n$ if and only if
there is global survival on the set $\gamma$.
Suppose now that the ray is chosen in such a way that
\[
 \max\left ( p(x_0,x_1), \sup_{i \ge 1} (p(x_i,x_{i+1})+p(x_i,x_{i-1})) \right )= \alpha <1;
\]
by choosing $\bar \rho \in \left (1, \min(\theta^{-1}, \alpha^{-1}) \right )$ we have, by induction on $n$,
that $\sum_{y \in \gamma} m^{(n)}(x,y) \le \bar \rho^n \alpha^n$ and then, according
to Theorem~\ref{th:equiv1}(4), there is no global survival on $\gamma$. This implies
that there is no global survival on $X_n$ for any $n \in \N$.
An explicit example is given by the homogeneous tree with degree
larger than 3 with the simple random walk.
\end{rem}

A possible application of Theorem~\ref{th:spatial} is based on Definition~\ref{def:invariantmu}.
 Consider 
an injective map $K$ and suppose that $\mu$ is $K$-invariant.
If $(X,E_\mu)$ is connected and there exists $Y \subseteq X$ such that for
all finite subsets $A \subset X$ we have
$X_n:=K^{(n)}(Y) \supseteq A$ (for some $n \in \N$)
then the BRW$(X,\mu)$ survives (locally) wpp if and only if
$(Y,\nu)$ survives (locally) wpp (where $\nu$ is the law induced
by $\mu$ on $Y$). Indeed, note that, if $\mu_n$ is the law induced by $\mu$ on $X_n$,
since $\mu$ is $K$-invariant, the behavior of $(X_n,\mu_n)$
is the same as the behavior of $(Y,\nu)$.
Clearly if  $(Y,\nu)$  survives then $(X,\mu)$ survives.
On the other hand if $(X,\mu)$ survives, according to Theorem~\ref{th:spatial}, there exists
a finite subset $A \subset X$ such that the induced BRW on $A$ survives thus,
for any $n$ such that $X_n \supseteq A$ 
we have that $(X_n, \mu_n)$
survives and this implies the survival of $(Y,\nu)$.

This applies, in particular, when $X=\Z^d$, $\mu_x$ is translation invariant and
$Y$ is a cone, namely $Y=\{y \in \Z^d :\langle y, y_0 \rangle \geq \alpha \|y\|\cdot\|y_0\| \}$ for some fixed nontrivial  $y_0 \in \Z^d$
and $\alpha <1$ (where $\langle \cdot, \cdot \rangle$ and $\|\cdot\|$ represent the
usual scalar product and norm of $\Z^d$ respectively). In this case $K(x):=x-y_0$.
Roughly speaking, this means that if the BRW survives locally on $\Z^d$ it must survive
even when it is restricted to a proper subset $Y$ as long as the collection of
all the ``$K$-translations'' of $Y$ cover the whole space $\Z^d$.


\section{Approximation by truncated BRWs}
\label{sec:truncated}

In this section we want to study the approximation of a BRW $\{\eta_n\}_{n \in \N}$
by means of the sequence of truncated BRWs $\{\{\eta_n^m\}_{n \in \N}\}_{m \in \N}$.
We already know that if the BRW dies out locally (resp.~globally) a.s.~then any truncated BRW dies out locally (resp.~globally) a.s.~(this
can be proved by coupling as explained in Section~\ref{subsec:1}).
On the other hand we would like to be able to prove a result similar to Theorem~\ref{th:spatial}
as $m$ tends to infinity. For continuous-time BRWs this has been done in
\cite{cf:BZ3}; the technique we use here is essentially the same.
From now on the set $X$ is assumed to be countable; indeed, if it is finite
then there is no survival for the truncated BRW $\{\eta_n^m\}_{n \in \N}$
for any $m \in \N$.
Moreover, for technical reasons we suppose that the graph $(X,E_{\mu})$ has finite geometry, that is,
$\sup_{x \in X} \mathrm{deg}(x)<+\infty$. 

 In the following (see Step 3 below) we need to find
a measure $\rho$ which dominates stochastically all the measures $\{\rho_x\}_{x \in X}$
(where $\rho_x$ is the distribution of the number of children of a particle living at $x$,
see Section~\ref{subsec:discrete}). It is
straightforward to see that the existence of such a measure $\rho$ is equivalent to
$\sup_{x \in X} \rho_x([n,+\infty)) \to 0$ as $n \to +\infty$ (that we assume henceforth).
In this case $\rho$ can be chosen according to
\begin{equation}\label{eq:chooserho}
\rho(n)=\sup_{x \in X} \rho_x([n,+\infty))-\sup_{x \in X} \rho_x([n+1,+\infty)).
\end{equation}
Moreover the measure $\rho$ can be chosen with finite first (resp.~$k$-th)
moment if and only if \break
$\sum_{n \ge1} \sup_{x \in X} \rho_x([n,+\infty))<+\infty$
(resp.~$\int_0^\infty \sup_{x \in X} \rho_x([\sqrt[k]{t},+\infty)) \diff t<+\infty$).
%

We assume that the matrix $M$ is irreducible and we denote its
convergence parameter by $R_\mu$. We observe that, using this notation,
according to Theorem~\ref{th:equiv1}(1), local survival
is equivalent to $R_\mu<1$. Remember that, in this case,
$\liminf_{n\to\infty} \sqrt[n]{\sum_y m^{(n)}_{xy}}$ and $\limsup_{n\to\infty} \sqrt[n]{m^{(n)}_{xx}}$
do not depend on the choice of $x,y \in X$.

In the following, we need to define the product of two graphs (basically, these will
be space/time products):
given two graphs $(X,\mathcal E)$, $(Y,\mathcal E^\prime)$ we denote by
$(X,\mathcal E) \times (Y,\mathcal E^\prime)$ the weighted graph with set of vertices
$X \times Y$ and set of edges $\mathcal E=\{((x,y),(x_1,y_1)): (x,x_1) \in\mathcal E,
(y,y_1) \in \mathcal E^\prime\}$.
%

\subsection{The comparison with an oriented percolation}
\label{sec:roadmap}

First of all, remember the coupling between $\{\eta_n\}_{n\in \N}$ and $\{\eta^m_n\}_{n \in \N}$:
the truncated process $\{\eta^m_n\}_{n \in \N}$ (satisfying equation~\eqref{eq:evolBRWm}) can be seen as
the BRW $\{\eta_n\}_{n \in \N}$ (satisfying equation~\eqref{eq:evolBRW}) by removing, at each step, all the births
which cause more than $m$ particles to live on the same site.
As in \cite{cf:BZ3} we need two other coupled processes.
Fix $\widetilde n\in\N$ (which plays the same role as $n_0$ in \cite{cf:BZ3})
and 
let $\{\bar\eta_n\}_{n \in \N}$ be the process obtained from
the BRW $\{\eta_n\}_{n \in \N}$ by removing all $n$-th generation particles
with $n> \widetilde n$, that is
 \begin{equation}\label{eq:comparison}
 \bar \eta_n=
\begin{cases}
\eta_n & n \le \widetilde n\\
0 & n >\widetilde n.
\end{cases}
 \end{equation}
Define $\{\bar\eta^m_n\}_{n \in \N}$ analogously
from $\{\eta^m_n\}_{n \in \N}$. Clearly, the following stochastic inequalities hold
$\eta_n\ge\eta^m_n$
and
$\bar\eta_n\ge\bar\eta^m_n$ for all $n \in \N$.
By construction, the progenies of a given particle
in $\{\bar\eta_n\}_{n \in \N}$ or $\{\bar\eta^m_n\}_{n \in \N}$
lives at a distance from the ancestor not larger than $\widetilde n$.

Our proofs are essentially divided
into four main steps. We report here shortly the essence
of these steps and we refer to \cite[Section 4]{cf:BZ3} for further details.

\begin{step}\label{st:1}
Fix a graph $(I,\mathcal E(I))$ such that the
Bernoulli percolation on $(I,\mathcal E(I))\times\vec\N$
has a critical value $p_c<1$ (where we denote by  $\vec{\N}$ the oriented graph on $\N$, that is, $(i,j)$ is an edge if and only
if $j=i+1$).
\end{step}

The usual trick is to find a copy of the graph
$\Z$ or $\N$ as a subgraph of $I$, since the (oriented) Bernoulli bond percolation on $\Z \times \vec{\N}$ and
$\N \times \vec{\N}$ has two phases. In this paper, the
main choices for $I$ are $\Z$, $\N$ or
$X$.

\begin{step}\label{st:2}
Given a globally (or locally) surviving BRW
and for every $\eps>0$
there exists a collection of disjoint sets $\{A_i\}_{i\in I}$ ($A_i\subset X$ for all $i\in I$), $\bar n>0$,
and $k\in\N \setminus\{0\}$, such that, for all $i \in I$,
\begin{equation}\label{eq:step2}
\pr \Big ( \forall j:(i,j)\in \mathcal E(I),
\sum_{x\in A_j}\eta_{\bar n
}(x)\ge k\Big |\eta_{0}= \eta\Big )
>1-\eps,
\end{equation}
for all $\eta$ such that
$\sum_{x\in A_i}\eta(x)= k$ and $\eta(x)=0$ for all $x\not\in A_i$.
The same holds, for  $\widetilde n \ge \bar n$, for $\{\bar\eta_n\}_{n\in\N}$ in place of $\{\eta_n\}_{n\in\N}$.
\end{step}

In the following sections Step 2 will be established under certain conditions
(and for suitable choices of $(I, \mathcal E(I))$). Basically we have to prove that, for a suitable surviving
BRW, with a probability arbitrarily close to $1$, given enough particles in $A_i$, after a fixed time $\bar n$,
we have at least the same number of particles on every neighboring set $A_j$.

In order to understand Steps 2, 3 and 4,  we describe briefly the percolation that we want to construct.
Consider an edge $((i,l),(j,l+1))$ in $(I,\mathcal E(I))\times\vec\N$:
let it be open if $\{\eta_n^m\}_{n \in \N}$ has at least $k$ individuals in $A_i$ at time $l\bar n$ and
in $A_j$ at time $(l+1)\bar n$. Step 3, which follows from Step 2, guarantees that, if $m$ is sufficiently large,
then all the edges exiting
a fixed vertex are open with a probability arbitrarily close to 1.
Thus the probability of global survival of $\{\eta_n^m\}_{n \in \N}$ is bounded from below by
the probability that there exists an infinite cluster containing $({i_0},0)$
in this percolation on $I\times\vec\N$, and, if $A_{i_0}$ is finite,
the probability of local survival
is bounded from below by the probability that the cluster contains infinitely many points
in $\{({i_0},l):l\in\N\}$ (we suppose to start with $k$ particles in $A_{i_0}$).
Let $\nu_1$ be the associated percolation measure.
Unfortunately this percolation is neither independent nor one-dependent.
In fact the opening procedure of the edges $((i,n),(j,n+1))$
and $((i_1,n),(j_1,n+1))$ may depend respectively on two different
progenies of particles overlapping on a vertex $x_0$. This may cause
dependence since if in $x_0$ there are already $m$ particles then
newborns are not allowed. Step 4
overcomes this problem.

\begin{step}\label{st:3}
Let $\eps$, $\{A_i\}_{i\in I}$, $\bar n$ and $k$ be chosen as in Step \ref{st:2}.
Then for all sufficiently large $m$ we have that,
for all $i \in I$,
\begin{equation}\label{eq:step3}
\pr \Big ( \forall j:(i,j)\in \mathcal E(I),
\sum_{x\in A_j}\eta^m_{\bar n
}(x)\ge k\Big |\eta^m_{0}= \eta\Big )
>1-2\eps,
\end{equation}
for all $\eta$ such that
$\sum_{x\in A_i}\eta(x)= k$,  $\eta(x)=0$ for all $x\not\in A_i$.
The same holds, for $\widetilde n \ge \bar n$, for $\{\bar\eta^m_n\}_{n\in\N}$ in place of $\{\eta^m_n\}_{n\in\N}$.
\end{step}
Step \ref{st:3} follows from Step \ref{st:2}; the proof is a natural
adaptation of the same arguments of \cite[Step 3]{cf:BZ3}.
Indeed let $N_n$ be the total number of particles ever born in the BRW before time $n$
(starting from the configuration $\eta$); it is clear that $N_n$ is a 
process stochastically dominated (the arguments are similar to the ones we used in Section~\ref{subsec:1}) by a 
branching process with offspring law
%
%
%
\[
\rho^\prime(n)
:=
\begin{cases}
0 & n=0\\
\rho(n-1) & n \ge 1
\end{cases}
\]
and initial state $N_0$ (where $\rho$ is given by equation~\eqref{eq:chooserho}).
If $N_0<+\infty$ almost surely then for all $n>0$ we have
$N_n<+\infty$ almost surely; hence for all  $n>0$, $k >0$ and $\eps>0$
there exists $N(n,\eps,k)$
such that, for all $i \in I$,
\[
\pr\Big (N_n\le N(n,\eps,k)\Big |\eta_{0}= \eta\Big )>1-\eps,
\]
for all $\eta$ such that
$\sum_{x\in A_i}\eta(x)= k$, $\eta(x)=0$ for all $x\not\in A_i$.
Define $\widetilde N=N(\bar n,\eps,k)$:
in order to compare with \cite[Step 3]{cf:BZ3}, we observe that
$\widetilde N=N(\bar n,\eps,k)$ plays the same role played
by $\bar n=n(\bar t, \varepsilon)$ in \cite{cf:BZ3}.
The conclusion follows, using elementary probability arguments, as in \cite[Step 3]{cf:BZ3}
by choosing $m \ge \widetilde N$.

\begin{step}\label{st:4}
Given a globally (or locally) surviving BRW,
for every $\eps>0$ and for all sufficiently large $m$,
there exists a  one-dependent oriented percolation
on $I\times\vec\N$ (with probability $1-2\eps$ of opening simultaneously all edges from a vertex
and $2\varepsilon$ of opening no edges)
such that the probability of survival of the BRW$_m$ (starting at time $0$ from a configuration
$\eta$ such that
$\sum_{x\in A_{i_0}}\eta(x)= k$ and $\eta(x)=0$ for all $x\not\in A_{i_0}$)
is larger than the probability
that there exists an infinite cluster containing $({i_0},0)$.
\end{step}

To overcome the dependence of the ``natural'' percolation that we described above,
we adapt \cite[Step 4]{cf:BZ3} to a discrete-time process:
the construction is made by means of the process $\{\bar \eta^m_n\}_{n\in\N}$ by choosing
$m \ge 2 \widetilde N H$ where $\widetilde N$ is the same as in Step 3 and
$H$ is the maximum of the number of paths of length
$\widetilde n$ crossing a vertex (the assumption of bounded geometry that we made on the graph
plays a fundamental
role here). Note that with this choice of $m$ we have that, starting with an initial condition
$\eta_0$ (such that $\sum_{x \in X} \eta_0(x)=k$),
$\eta^m_n=\eta_n$ and $\bar \eta^m_n=\bar \eta_n$ for all $n \le \bar n$ on an event with probability at least
$1-\varepsilon$ (namely, $\{N_{\bar n} \ge N(\bar n,\varepsilon,k)\}$ as defined in Step 3).
Step \ref{st:4} follows then from Step \ref{st:3}.


\medskip

Our next goals are to fix suitable graphs $(I,\mathcal E(I))$ and prove
Step \ref{st:2} for a large class of globally surviving BRWs: then by Steps \ref{st:4} and \ref{st:1},
for all sufficiently large $m$, the corresponding truncated BRW$_m$ survives globally wpp
if $m$ is sufficiently large.
On the other hand, in order to show that, given a \textsl{locally} surviving BRW,
the corresponding truncated BRW$_m$ survives wpp
if $m$ is sufficiently large,
we need to prove Step \ref{st:2} with a choice of at least one $A_i$ finite,
say $A_{i_0}$, and $I$ containing a copy of $\Z$ or $\N$
as a subgraph.
Remember that, in a supercritical Bernoulli bond percolation in
$\Z \times \vec{\N}$ or $\N \times \vec{\N}$, with probability 1 the infinite open cluster
has an infinite intersection with the set $\{(0,n): n \in \N\}$.
Thus, in the supercritical case we have, wpp,
in the infinite open cluster, 
an infinite number of vertices of the set $\{(0,n): n \ge 0\}$ including the origin.
This (again by Steps \ref{st:3} and \ref{st:4}) implies that, wpp, the BRW$_m$ starting with $k$ particles in
$A_{i_0}$ has particles alive in $A_{i_0}$ at arbitrarily large times.
Being $A_{i_0}$ finite yields the conclusion.

\begin{rem}\label{rem:roadmap2}
As in \cite{cf:BZ3}, the previous set of steps represents the skeleton of the proofs of Theorems \ref{th:main} and \ref{th:tree}.
In order to be able to prove Theorem \ref{th:zdrift} we need to modify this approach.
Here are the main differences. We choose an oriented graph
$(W, {\mathcal E}(W))$ and a family of subsets of $X$,
$\{A_{(i,n)}\}_{(i,n) \in W}$ such that
\begin{itemize}
\item $W$ is a subset of the set $\Z \times \N$ (the inclusion is between sets not between graphs);
\item for all $n \in \N$ we have that $\{A_{(i,n)}\}_{i:(i,n) \in W}$ is a collection of disjoint subsets of $X$;
\item $(i,n) \to (j,m)$ implies $m=n+1$.
\end{itemize}
Step \ref{st:2} translates into the following:
given a (globally or locally) surviving BRW
and for every $\eps>0$,
there exists $\bar n>0$ and
$k\in\N$, such that, for all $n \in \N$, $i \in \Z$, and for all $\eta$ such that
$\sum_{x\in A_{(i,n)}}\eta(x)= k$,
\[
\pr \Big (\forall j:(i,n) \to (j,n+1),
\sum_{x\in A_{(j,n)}}\eta_{(n+1) \bar n
}(x)\ge k\Big |\eta_{n \bar n}=\eta\Big )
>1-\eps.
\]
Step \ref{st:3} is the same and the percolation in Step \ref{st:4} now concerns
the graph $(W, {\mathcal E}(W))$.
\end{rem}

\subsection{Local survival}\label{subsec:local}

Let us choose a vertex $o \in X$, fix the initial configuration
as $\eta_0:=\delta_o$ and
assume that the measure $\rho$ as defined by equation~\eqref{eq:chooserho}
has finite second moment.
The key to prove Step 2 is based on some estimates on the
expected value $\E^{\delta_o}(\eta_n(x))$ of the number of individuals in a site.
This expected value can be computed using
equation~\eqref{eq:evolexpectednumber}: hence
$\E^{\delta_o}(\eta_n(x))=m^{(n)}_{o,x}$. It is clear that
\begin{equation}
\label{eq:estimateE}
\lim_{n \to \infty} \E^{\delta_o}(\eta_n(x))
=
\begin{cases}
 0 & \text{if } R_\mu >1,\\
+\infty & \text{if } R_\mu < 1.\\
\end{cases}
\end{equation}

In the following lemma we prove that, when $R_\mu<1$, if at time 0 we have one
individual at each of the sites $x_1,\ldots,x_l$, then, given any choice
of $l$ sites $y_1,\ldots, y_l$, after some time the expected
number of descendants in $y_i$ of the individual in $x_i$
exceeds any fixed $D \ge 1$ for all $i=1,\ldots,l$. The proof follows immediately from equation~\eqref{eq:estimateE}
we omit it; just note that, due to equation~\eqref{eq:comparison}, the estimate on $\E^{\delta_{x_j}}(\bar\eta_n(y_j))$
follows immediately from the one on $\E^{\delta_{x_j}}(\eta_n(y_j))$. 

\begin{lem}\label{lem:2.1.5}
Let us consider the finite set of couples
$\{(x_j,y_j)\}_{j=0}^l$ and fix $D \ge 1$; if $R_\mu<1$
then there exists $n>0$ 
such that
$\E^{\delta_{x_j}}(\eta_n(y_j)) > D$, 
$\forall j=0,1,\ldots,l$. Moreover, $\E^{\delta_{x_j}}(\bar\eta_n(y_j)) > 1$
when $\widetilde n>n$.
\end{lem}

%

We show
that, when $R_\mu<1$, for all sufficiently large $k \in \N$, given $k$ particles in a site $x$ at time 0,
``typically'' (i.e.~with arbitrarily large probability)
after some time we will have at least $Dk$ individuals in each site
of a fixed finite set $Y$.
Analogously, starting with $l$ colonies of size $k$ (in sites
$x_1,\ldots,x_l$ respectively), each of them will spread, after a
sufficiently long time, at least $Dk$ descendants
in every site of a corresponding (finite) set of sites $Y_i$.

\begin{lem}\label{th:peps} Suppose that $R_\mu<1$.
\begin{enumerate}
\item  
Let us fix $x\in X$, $Y$ a finite subset
of $X$, $D \ge 1$ and $\eps>0$.
Then there exists $\bar n= \bar n(x, Y)>0$ (independent of $\eps$),
$k(\eps,x,Y)$ such that, for all $k \ge k(\eps,x,Y)$,
\[
\pr \left(\bigcap_{y \in Y}\{\eta_{\bar n}(y)\ge Dk\} \Big | \eta_0(x)=k
\right)>1-\eps.
\]
The claim holds also with $\{\bar\eta_n\}_{n\in\N}$ in place of $\{\eta_n\}_{n\in\N}$ when $\widetilde n \ge \bar n$.
\item Let us fix a finite set of vertices $\{x_i\}_{i=1, \ldots,m}$, a collection
of finite sets $\{Y_i\}_{i=1, \ldots,l}$ of vertices of $X$, $D \ge 1$ and $\eps>0$.
Then there exists $\bar n=\bar n(\{x_i\},\{Y_i\})$ (independent of $\eps$), $k(\eps,\{x_i\},\{Y_i\})$
such that, for all $i=1, \ldots,l$ and $k \ge k(\eps,\{x_i\},\{Y_i\})$,
\[
\pr\left(\bigcap_{y \in Y_i}\{\eta_{\bar n}(y)\ge Dk\}
\Big | \eta_{0}(x_i)=k
\right)>1-\eps.
\]
The claim holds also with $\{\bar\eta_n\}_{n\in\N}$ in place of $\{\eta_n\}_{n\in\N}$ when $\widetilde n \ge \bar n$.
\end{enumerate}
\end{lem}

\begin{proof} 
\begin{enumerate}
\item
If we denote by $\{\{\xi_{n,i}\}_{n\in\N}\}_{i \in \N}$ a family of independent BRWs behaving according
to $\mu$ and starting from $\xi_{0,i}=\delta_x$ (for all $i \in \N$)
then,
by Lemma~\ref{lem:2.1.5},
we can choose $\bar n$ such that $\E^{\delta_x}(\xi_{\bar n,i}(y))>2D$ for all $y \in Y$.
Consider a realization of $\{\eta_n\}_{n\in\N}$ such that $\eta_n(y)=\sum_{j=1}^k\xi_{n,j}(y)$; 
denote the variance $\var(\xi_{n,j}(y))$ by $\sigma^2_{n,y}$.
Since $\xi_{n,j}$ is stochastically dominated by a branching process with offspring law $\rho$
(where $\rho$ is chosen as in equation~\eqref{eq:chooserho}), it is clear that, for all $y$, $\sigma^2_{n,y}<\E(\rho)^{n-1}\var(\rho)<+\infty$
since we assumed at the beginning of this section that $\rho$ has finite second moment.
%
By using the one-sided Chebyshev inequality
\[
\begin{split}
\pr\left(\eta_{\bar n}(y)\ge Dk
\right)& \ge \pr \big ( \eta_{\bar n}(y) \ge \E(\eta_{\bar n}(y))/2 \big )
\ge \frac{\E(\eta_{\bar n}(y))^2/4}{\E(\eta_{\bar n}(y))^2/4+k\sigma^2_{\bar n,y}}
\ge 1-\frac{\sigma^2_{\bar n,y}}{D^2 k + \sigma^2_{\bar n,y}}\\
\end{split}
\]
Hence, fixed any $\delta >0$, there exists $k(\delta,x,y)$ such that, for all $k \ge k(\delta,x,y)$,
$\pr\left(\eta_{\bar n}(y)\ge Dk
\right)\ge
1-
\delta
$.
For all $k \ge \max_{y\in Y}k(\delta,x,y)<+\infty$
\[\pr \left (\bigcap_{y\in Y}(\eta_{\bar n}(y)\ge Dk) \Big |\eta_0(x)=k
\right)
\ge 1-2|Y|\delta,
\]
where $|Y|$ is the cardinality of $Y$.
The assertion for $\bar\eta_n$ follows from Lemma \ref{lem:2.1.5}.
\item
Let $\{\{\xi_{n,i}\}_{t \ge 0}\}_{i \in \N}$ be as before and choose $\bar n$ such that $\E^{\delta_{x_i}}(\xi_{\bar n,i}(y))>2D$ for all $y \in Y_i$ a
nd for all $i=1, \ldots l$.
According to (1) above we may fix $k_{i}$ such that, for all $k \ge k_{i}$,
\[
\pr\left(\bigcap_{y\in Y_i}\{\eta_{\bar n}(y)\ge Dk\} \Big |\eta_0(x_i)=k
\right)
\ge 1-\eps.
\]
Take $k \ge \max_{i=1, \ldots, l} k_{i}$ to conclude. Again the assertion for $\bar\eta_n$ follows from Lemma \ref{lem:2.1.5}.
\end{enumerate}
\end{proof}



The dependence of $k$ on the offspring distribution $\mu$ is hidden in the term
$\sigma^2_{\bar n,y}$, that is, in $\bar n$ and in the dominating offspring law $\rho$.
The key is to find a fixed $k$ such that the lower bound in the previous theorem
holds simultaneously for a family $\{(x_i,Y_i)\}$. One possibility is to choose
a finite family (as we did in the previous lemma) but it is not the only one:
one has to find a fixed $\bar n$ such that Lemma~\ref{lem:2.1.5} holds
(for all the couples $(x_i,y)$ where $y \in Y_i$) and this gives immediately an
upper bound for $\sigma^2_{\bar n, y}$ (uniform with respect to $y$).

\begin{rem}\label{rem:anh}
Note that Lemmas~\ref{lem:2.1.5} and \ref{th:peps}
can be restated for the process
$\{\bar{\eta}^m_n\}_{n\in\N}$ if $m$ is sufficiently large.
Indeed, when $m \ge 2 
N(\bar n, \varepsilon, Dk)H$ (as in Step 4) we have that
$\bar \eta^m_n= \bar \eta_n$ for all $n \le \bar n$ on an event with probability at least
$1-\varepsilon$.
In the rest of the paper, when not explicitly stated otherwise, Lemma~\ref{th:peps}
will be used by setting $D=1$.
\end{rem}

We already know that
if $\{\eta_n\}_{n \in \N}$ dies out locally (resp.~globally) a.s.~then $\{\eta^m_n\}_{n \in \N}$
dies out locally (resp.~globally) a.s. The following theorem states the converse.

We recall that $(X,\mu)$ is \textit{quasi transitive} if and only if there exists
a finite subset $X_0 \subseteq X$ such that for all $x \in X$ there exists a bijective
map $\gamma:X\to X$ and $x_0 \in X_0$ satisfying $\gamma(x_0)=x$ and $\mu$ is $\gamma$-invariant.

\begin{teo} \label{th:main}$ $\\
If at least one of the following conditions holds
\begin{enumerate}
\item  $(X,\mu)$ is 
 quasi transitive and connected;
\item $(X,\mu)$ is connected and there exists
$\gamma$
bijection on $X$ such that
\begin{enumerate}
\item
$\mu$ is $\gamma$-invariant;
\item for some $x_0\in X$ we have $x_0=\gamma^nx_0$ if and only if $n=0$;
\end{enumerate}
\end{enumerate}
then 
if $\{\eta_n\}_{n \in \N}$ survives locally (starting from $x_0$) then $\{\eta^m_n\}_{n \in \N}$ survives locally (starting from $x_0$) eventually as $m\to +\infty$.
\end{teo}
\begin{proof}
\begin{enumerate}
\item Let $R_\mu<1$ and define, for any $x \in X_0$, $Y_x:=\{y \in X : 
 (x,y) \in E_\mu\}$. Fix $I=X$, $\mathcal E(I)=\{(x,y):(x,y)\in  E_\mu\text{ or
}(y,x)\in  E_\mu\}$ and $A_x=\{x\}$.
Lemma \ref{th:peps} yields Step 2.
To prove that the percolation on $(I,\mathcal E(I))\times \vec{\N}$
has two phases
(that is, $(I,\mathcal E(I))$ can be used in Step 1) we note that
this follows 
from the fact that the graph $\mathbb{N}$ is a
subgraph of $X$. Recall that in the supercritical Bernoulli percolation on
$\N \times \vec{\N}$ wpp the infinite open
cluster contains $(0,0)$ and intersects the $y$-axis infinitely
often. 
Hence by Steps 3 and 4 we have that, for all sufficiently large $m$,
$\{\eta^m_n\}_{n \in \N}$ survives locally.

\item
By
Lemma \ref{th:peps}, there exists $\bar n$ such that, for sufficiently large $\widetilde n$,
\[
\begin{cases}
\pr\left(\bar\eta_{\bar n}(\gamma x_0)\ge k \Big | \bar\eta_{0}(x_0)=k
\right)>1-\eps \\ 
\pr\left(\bar\eta_{\bar n}(x_0)\ge k
\Big | \bar\eta_{0}(\gamma x_0)=k \right)>1-\eps. \\
\end{cases}
\]
This implies easily
\[
\begin{cases}
\pr\left(\bar\eta_{\bar n}(\gamma^{n} x_0)\ge k \Big |
\bar\eta_{0}(\gamma^{n-1}x_0)=k \right)>1-\eps \\ 
\pr\left(\bar\eta_{\bar n}(\gamma^{n-1} x_0)\ge k \Big |
\bar\eta_{0}(\gamma^{n} x_0)=k \right)>1-\eps \\
\end{cases}
\]
for all $n \in \mathbb{Z}$ since $\mu$ is $\gamma$-invariant.
Thus $\{\eta^m_n\}_{n \in \N}$ survives locally (for sufficiently large $m$) applying Step 3 and 4
(here $I=\mathbb{Z}$ and $A_i=\{\gamma^i x_0\}$).
\end{enumerate}

\end{proof}

%
%

\subsection{Global survival}\label{subsec:global}


In this section we discuss how the global behaviors of
$\{\eta^m_n\}_{n \in \N}$ and $\{\eta_n\}_{n \in \N}$ are related
and when the global survival of $\{\eta_n\}_{n \in \N}$ implies eventually the
global survival of $\{\eta^m_n\}_{n \in \N}$.

If  $(X,\mu)$ is
 quasi transitive and $\liminf_{n\to\infty} \sqrt[n]{\sum_y m^{(n)}_{xy}}= \limsup_{n\to\infty} \sqrt[n]{m^{(n)}_{xx}}$
then the global survival of  $\{\eta_n\}_{n \in \N}$ implies the global survival
of $\{\eta^m_n\}_{n \in \N}$ for a sufficiently large $m\in \N$.
Since a quasi-transitive BRW is an $\mathcal{F}-BRW$, according to Theorems~\ref{th:equiv1} and \ref{th:fgraphf}
$\{\eta_n\}_{n \in \N}$ survives globally if and only if it survives locally. We  proved in Theorem~\ref{th:main}
that $\{\eta^m_n\}_{n \in \N}$ survives locally (for sufficiently large $m$), thus it survives globally.

\begin{rem}\label{rem:coupling}
The basic idea of this section is to take a BRW $(X,\mu)$ which is
locally isomorphic to a BRW $(I,\nu)$ (the projection map being $g$);
we define $\{A_i\}_{i \in I}$ by $A_i:=g^{-1}(i)$.
We know that, if $\{\eta_n\}_{n \in \N}$ is a realization of $(X,\mu)$
then a realization of $(I,\nu)$ is given by the projection (on $I$) $\{\xi_n\}_{n\in\N}$ where
$\xi_n=\pi_g(\eta_n)$ for all $n \in\N$. Clearly $\nu_{g(x)}(\cdot)=\mu_x(g^{-1}(\cdot))$ and
we can easily
compute the expected number of particles alive at time $n$ in
$A_i$ starting from a single particle alive in $x$ at time $0$ as
\begin{equation}\label{eq:locisom} 
\sum_{z \in A_i} \E_\mu^{\delta_x}(\eta_n(z))=\E_\nu^{\delta_{f(x)}}(\xi_n(i)).
\end{equation}
Since
$\{\eta^m_n\}_{n \in \N}$ and $\{\pi_g(\eta^m_n)\}_{n \in \N}$ have the same global
behavior and $\{\pi_g(\eta^m_n)\}_{n \in \N}$ stochastically dominates
$\{\xi_n^m\}_{n \in \N}$ then if the latter survives globally wpp
then $\{\eta^m_n\}_{n \in \N}$  survives globally wpp.
\end{rem}

Following the previous remark, we take $I=\Z$, $X=\Z \times Y$ (for some set $Y$) and we denote by $g:X \to \Z$ the usual
projection from $X$ onto $\Z$, namely $g(n,y):=n$.

%

We suppose that $\nu$ is translation invariant 
and
we denote by $\rho$ and $\bar \rho=\sum_{y \in X} m_{xy} = \sum_{j \in \Z} \widetilde m_{g(x)j}$ the distribution and the expected number of offsprings
of $\{\eta_n\}_{n\in\N}$ respectively
(where, according to the notation of Section~\ref{sec:technical}, $\widetilde m_{ij}$ is the
expected number of offsprings in $j$ of a particle in $i$ of the projected BRW
$\{\xi_n\}_{n\in\N}$). We note that, since $\rho$ and $\bar \rho$ are the distribution and the
expected number of offsprings of $\nu$ as well, they do not depend on $x \in X$ or $i \in \Z$
since $\nu$ is translation invariant. The following Theorem is the discrete-time analogous of
\cite[Theorem 6.1]{cf:BZ3}.

\begin{teo}\label{th:zdrift}
Let $X=\Z \times Y$ and suppose that the BRW $(X, \mu)$ is locally isomorphic to $(\Z, \nu)$ where  $\nu$ is
translation invariant.
If $m_{xy}=0$ whenever $|g(x)-g(y)|>1$ then
\begin{enumerate}
\item
the BRW survives globally starting from $x$ if and only if $\bar \rho=\sum_{y \in \Z} m_{xy} > 1$;
\item
if the BRW survives globally (starting from $x$) then
$\{\eta_n^m\}_{n \in \N}$ survives globally (starting from $x$) provided that $m$ is sufficiently
large.
\end{enumerate}
\end{teo}

\begin{proof}
\begin{enumerate}
\item
This follows from Theorem~\ref{th:fgraphf} since
$(X,\mu)$ is an $\mathcal F$-BRW which can be mapped onto the branching processes
with offspring distribution $\rho$ and recalling that $\sum_{y \in \Z} m^{(n)}_{xy}=\bar \rho^n$.
\item
According to Remark~\ref{rem:coupling} it 
is enough to prove
the claim for the BRW $\{\xi_n\}_{n\in\N}$ where
$\xi_n=\pi_g(\eta_n)$ whose diffusion matrix satisfies
\[
\widetilde m_{ij}=
\begin{cases}
p & j=i+1\\
q & j=i-1\\
1-p-q & i=j\\
0 & \text{otherwise}.
\end{cases}
\]
for some
$p,q \in [0,1]$ ($p+q \le 1$).

Here is an alternative proof which does not involve Remark~\ref{rem:coupling}.
Following the general version of Step 2
(see Remark \ref{rem:roadmap2}),  we define $A_{(i,n)}=g^{-1}(i)$
and we fix $\alpha,\beta\in(0,1)$ such that  $\alpha<\beta<(1+\alpha)/2$.
Note that
\begin{equation}\label{eq:lowerestimate}
\begin{split}
\widetilde p^{\,(n)}(0,\alpha n)&
=\sum_{i=\alpha n}^{(1+\alpha)n/2}\binom{n}{i,\ \ i-\alpha n,\ \ n-2i+\alpha n}p^{i}
q^{i-\alpha n}(1-p-q)^{n-2i+\alpha n}
\\
&\ge \binom{n}{\beta n,\ \ (\beta-\alpha)n,\ \ (1-2\beta+\alpha)n}p^{\beta n}
q^{(\beta-\alpha)n}(1-p-q)^{(1-2\beta+\alpha)n}\\
&\stackrel{n\to\infty}{\sim} \frac{1}{2\pi n\sqrt{\beta(\beta-\alpha)(1-2\beta+\alpha)}}\,
\left(\frac{p^\beta q^{\beta-\alpha}(1-p-q)^{1-2\beta+\alpha}}{\beta^\beta(\beta-\alpha)^{\beta-\alpha}
(1-2\beta+\alpha)^{1-2\beta+\alpha}}\right)^n
\end{split}
\end{equation}
(to avoid a cumbersome notation we write $n \alpha$ instead of
$\lfloor n \alpha \rfloor$).

Define
\[
Q_{ \bar \rho}(\alpha,\beta)=\frac{\bar \rho p^\beta q^{\beta-\alpha}(1-p-q)^{1-2\beta+\alpha}}{\beta^\beta(\beta-\alpha)^{\beta-\alpha}
(1-2\beta+\alpha)^{1-2\beta+\alpha}};
\]
if the BRW survives globally then $\bar \rho>1$ and
equation \eqref{eq:lowerestimate} implies
\[
\begin{split}
\E^{\delta_0}(\xi_n(\alpha n))&=\bar \rho^n p^{(n)}(0, n\alpha)\\
&\ge \bar \rho^n \binom{n}{\beta n,\ \ (\beta-\alpha)n,\ \ (1-2\beta+\alpha)n}p^{\beta n}
q^{(\beta-\alpha)n}(1-p-q)^{(1-2\beta+\alpha)n}\\
&\sim
\frac{1}{2\pi n
\sqrt{\beta(\beta-\alpha)(1-2\beta+\alpha)}}\,
\left(Q_{\bar \rho}(\alpha,\beta)\right)^n
\end{split}
\]
as $n \to \infty$. This, along with equation \eqref{eq:locisom}, implies easily that
$\sum_{x \in A_{\alpha n}}\E^{\delta_0}(\eta_n(x))$ has a lower bound which is
asymptotic to $\frac{1}{(2\pi n)
\sqrt{\beta(\beta-\alpha)(1-2\beta+\alpha)}}\,
\left(Q_{\bar \rho}(\alpha,\beta)\right)^n$ as $n \to \infty$.

Note that $Q_{\bar \rho}(p-q,p)=\bar\rho>1$, thus there exist $\alpha_1<\alpha_2\le\beta_1<\beta_2$
(with $\beta_i<(1+\alpha_i)/2$, $i=1,2$) such
that $Q_{\bar \rho}(x,y)>1$, for all $(x,y)\in [\alpha_1,\alpha_2]\times
[\beta_1,\beta_2]$.
By taking $n=\widetilde N$ sufficiently large one can find three distinct integers
$d_1$, $d_2$ and $d_3$ such that $\alpha_1 n\le d_1<d_2\le\alpha_2n$,
$\beta_1 n\le d_3\le\beta_2n$ and
$Q_{\bar\rho}(d_l/n,d_3/n)>1$, $l=1,2$.

By reasoning as in Lemma~\ref{th:peps} we have that, for all $\eps>0$,
there exists $\bar n$, $k=k(\eps)$ such that, for all $i\in\Z$, for all $\widetilde n$
sufficiently large,
\[
\pr \left(
\sum_{x \in A_{i+j}} \bar\eta_{\bar n}(x)\ge k,j=d_1,d_2
\Big |
\bar\eta_{0}(i)= \eta \right )
>1-\eps
\]
$\forall i \in \Z$ and for all $\eta$ such that $\sum_{x \in A_{i}} \bar\eta(x)= k$.
Since $k$ and $\bar n$ are independent of $i$ we have proven the general version of Step 2
using $W=\{a(d_1,1)+b(d_2,1):a,b\in\N\}$ where $(i,n) \to (j,n+1)$ if and
only if $j-i=d_1$ or $j-i=d_2$.
\end{enumerate}
\end{proof}

The previous theorem applies to translation invariant  BRWs on two particular graphs:
$\Z^d$ and the homogeneous tree $\mathbb{T}_{r}$ with degree $r$.
The following two corollaries are the discrete-time analogous of
\cite[Corollary 6.1]{cf:BZ3} and \cite[Theorem 6.2]{cf:BZ3}

\begin{cor}\label{cor:zd}
If the BRW $(\Z^d,\mu)$ is translation invariant and there exists a projection $g$ on one of the coordinates
such that $m_{xy}=0$ whenever $|g(x)-g(y)|>1$, then
\begin{enumerate}
\item
the BRW survives globally (starting from $x$) if and only if $\bar \rho=\sum_{y \in \Z} m_{xy} > 1$;
\item
if the BRW survives globally (starting from $x$)
$\{\eta_n^m\}_{n \in \N}$ survives globally (starting from $x$) provided that $m$ is sufficiently
large.
\end{enumerate}
\end{cor}

\begin{proof}
If $d=1$ then the proof is trivial.
If $d>1$, the claim follows immediately from the fact that,
since $\mu$ is translation invariant, then $(X,\mu)$ is locally isomorphic
to $(\Z,\nu)$ where the projected measure $\nu$ is translation invariant.
\end{proof}

\begin{cor}\label{th:tree}
Let $\mathbb{T}_r$ be a homogeneous tree and suppose that the BRW $(\mathbb{T}_r,\mu)$ is
$\gamma$-invariant for every automorphism $\gamma$ of
 $\mathbb{T}_r$. 
If
$\mu_x(f)\not = 0$ implies $\mathrm{supp}(f) \subseteq B(x,1)$ (where $B(x,1)$
is the usual ball of radius 1 and center $x$ of the graph $\mathbb{T}_r$)
then
\begin{enumerate}
\item
the BRW survives globally (starting from $x$) if and only if $\bar \rho=\sum_{y \in \Z} m_{xy} > 1$;
\item
if the BRW survives globally (starting from $x$) then
$\{\eta_n^m\}_{n \in \N}$ survives globally provided that $m$ is sufficiently
large.
\end{enumerate}
\end{cor}

\begin{proof}
Fix an end $\tau$ in $\mathbb{T}_r$ and a root $o\in X$ and define the map
$h:X\to \Z$ as the usual height (see \cite[Section 12.13]{cf:Woess}).
Let $A_k=h^{-1}(k)$ as $k\in\Z$ (these sets are usually referred
to as horocycles). Since $\mu$ is invariant with respect to every
automorphism then we have, as before, that $\mathbb{T}_r=\Z \times \Z$ is
locally isomorphic to $(\Z,\nu)$ where the projection $\nu$ is translation invariant.

\end{proof}

\section{Final remarks}
\label{sec:open}

The paper is devoted to three main issues: finding conditions for the local (resp.~global)
survival of the process, discussing the spatial approximation and, finally, studying the
approximation by means of truncated BRWs.
This has been done for continuous BRWs in \cite{cf:BZ, cf:BZ2, cf:BZ3}.

About the first issue, a question was left open in \cite{cf:BZ2}, namely if
$\liminf_{n\to\infty} \sqrt[n]{\sum_y m^{(n)}_{xy}}>1$ implies global
survival starting from $x$. This suggests a more general question:
does the global behavior depend only on the first-moment matrix $M$?
We know by Theorem~\ref{th:equiv1}(1) that the local behavior of a discrete-time BRW depends only
on $M$.
Moreover, in \cite{cf:BZ2} has been proved that for a continuous-time
BRW there is a characterization in terms of a functional inequality of the global survival
and this inequality depends only on the
matrix $M$, namely
\begin{equation}\label{eq:functeq}
 \exists v \in [0,1] : v(x_0)>0, \ Mv \ge \frac{v}{\mathbf{1}-v}
\end{equation}
if and only if there is global survival starting from $x_0$
(again, the ratio in the right hand side is taken coordinatewise).
Hence, for a continuous-time BRW, both the local and global behaviors are completely determined by
the first-moment matrix of the process.
Nevertheless, according to Example~\ref{exm:noext}, one cannot expect to find an equivalent condition to global survival
for a discrete-time BRW involving only the first-moment matrix $M$.
Of course it is still possible to find either sufficient or necessary conditions for
the global survival which are similar to equation~\eqref{eq:functeq},
but this goes beyond the 
aim of this paper.

More work can still be done in the direction of understanding
the strong local survival by using the infinite-dimensional
generating function $G$; what we did in Section~\ref{subsec:stronglocal}
is just a beginning.

As for the spatial approximation, the results of Section~\ref{sec:spatial} are quite satisfactory. On the other
hand, there is room for improvements in the approximations by truncated BRWs of Section~\ref{sec:truncated}.
Indeed, one can hope to find more
classes of BRWs which can be approximated by their truncations. In our results a key role was played by
the similarity of the BRW under suitable automorphisms of the graph (such as translations, for instance), nevertheless
the four steps described in Section~\ref{sec:roadmap} (see also \cite[Section 5]{cf:BZ3})
are quite general and could be applied to a variety of classes
of BRWs, provided one can prove Step 2
(as we did in Sections~\ref{subsec:local} and \ref{subsec:global}, possibly using different techniques).

Finally, some of our results can be applied in a natural way to BRWs in random environment (as in \cite[Section 7]{cf:BZ3}) but,
again, this goes beyond the purpose of the paper.

\section*{Acknowledgments}
The author is grateful to the referees for many helpful suggestions.


\begin{thebibliography}{15}



\bibitem{cf:AthNey}
K.B.~Athreya,  P.E.~Ney,
Branching processes,
Die Grundlehren der mathematischen Wissenschaften, \textbf{196},
Springer-Verlag, 1972.


\bibitem{cf:BZ}
D.~Bertacchi, F.~Zucca,
Critical behaviors and critical values of branching random walks
on multigraphs, J.~Appl.~Probab.~\textbf{45} (2008), 481-497.

\bibitem{cf:BZ2}
D.~Bertacchi, F.~Zucca,
Characterization of the critical values of branching random walks on
weighted graphs through infinite-type branching processes,
J.~Stat.~Phys.~\textbf{134} n.~1 (2009), 53-65.

\bibitem{cf:BZ3}
D.~Bertacchi, F.~Zucca,
Approximating critical parameters of branching random walks,
J.~Appl.~Probab.~\textbf{46} (2009), 463-478.

\bibitem{cf:Big1977}
J.D.~ Biggins,
Martingale convergence in the branching random walk,
J.~Appl.~Probab.~\textbf{14}  n.~1 (1977), 25--37.

\bibitem{cf:Big1978}
J.D.~ Biggins,
The asymptotic shape of the branching random walk,
Adv.~Appl.~Probab.~\textbf{10} n.~1 (1978), 62--84.


\bibitem{cf:BigKypr97}
J.D.~ Biggins, A.E.~Kyprianou,
Seneta-Heyde norming in the branching random walk,
Ann.~Probab.  \textbf{25} n.~1 (1997), 337--360.

\bibitem{cf:BigRah05}
J.D.~ Biggins, A.~Rahimzadeh Sani,
Convergence results on multitype, multivariate branching random walks,
Adv.~Appl.~Probab.  \textbf{37} n.~3 (2005), 681--705.

\bibitem{cf:CMP98}
F.~Comets, M.V.~Menshikov, S.Yu.~Popov,
One-dimensional branching random walk in random environment: A classification,
Markov Process.~Related Fields~\textbf{4} (1998), 465--477.

\bibitem{cf:GW1875}
F.~Galton, H.W.~Watson, On the probability of the extinction of
families, Journal of the Anthropological Institute of Great Britain and Ireland \textbf{4} (1875), 138--144.

\bibitem{cf:GMPV09}
N.~Gantert, S.~M\"uller, S.Yu.~Popov, M.~Vachkovskaia,
Survival of branching random walks in random environment,
to appear on J.~Theor.~Probab., arXiv:0811.1748v3.

\bibitem{cf:GDH92}
A.~Greven, F.~den Hollander,
Branching random walk in random environment: Phase transitions for
local and global growth rates,
Probab.~Theory Related Fields~\textbf{91} (1992), 195--249.

\bibitem{cf:Harris63}
T.E.~Harris, The theory of branching processes, Springer-Verlag, Berlin, 1963.

\bibitem{cf:DHMP99}
F.~den Hollander, M.V.~Menshikov, S.Yu.~Popov,
A note on transience versus recurrence for a branching random walk in random environment,
J.~Stat.~Phys.~\textbf{95}  (1999), 587--614.

\bibitem{cf:HuLalley}
I.~Hueter, S.P.~Lalley, {Anisotropic branching random walks
on homogeneous trees},  Probab.~Theory Related Fields  \textbf{116},
(2000),  n.1, 57--88.

\bibitem{cf:Kin1965}
J.F.C.~Kingman,
Stationary measures for branching processes,
Proc.~Amer.~Math.~Soc.~\textbf{16} (1965), 245--247.

\bibitem{cf:Ligg1}
T.M.~Liggett,
{Branching random walks and contact processes on homogeneous trees},
Probab.~Theory Related Fields  \textbf{106},  (1996),  n.4, 495--519.

\bibitem{cf:Ligg2}
T.M.~Liggett,
{Branching random walks on finite trees},
Perplexing problems in probability,  315--330, Progr.~Probab.,
\textbf{44}, Birkh\"auser Boston, Boston, MA, 1999.

\bibitem{cf:MP00}
F.P.~Machado, S.Yu.~Popov,
One-dimensional branching random walk in a Markovian random environment,
J.~Appl.~Probab.~\textbf{37} (2000), 1157--1163.

\bibitem{cf:MP03}
F.P.~Machado, S.Yu.~Popov,
Branching random walk in random environment on trees,
Stochastic Process. Appl. \textbf{106} (2003), 95--106.

\bibitem{cf:MadrasSchi}
N.~Madras, R.~Schinazi, {Branching random walks on trees}, Stoch.~Proc.~Appl.
\textbf{42},  (1992),  n.2, 255--267.

\bibitem{cf:MountSchi}
T.~Mountford, R.~Schinazi, A note on branching random walks on
finite sets, J.~Appl.~Probab.~\textbf{42} (2005),  287--294.

\bibitem{cf:M08}
S.~M\"uller,
A criterion for transience of multidimensional branching random walk
in random environment,
Electron.~J.~Probab.~\textbf{13} (2008).

\bibitem{cf:PemStac1}
R.~Pemantle, A.M.~Stacey, {The branching random walk and
contact process on Galton--Watson and nonhomogeneous trees},
Ann.~Prob.~\textbf{29}, (2001),
 n.4, 1563--1590.

\bibitem{cf:Sen}
        E.~Seneta, {Non-negative matrices and Markov chains},
        Springer Series in Statistics, Springer, New York, 2006.

\bibitem{cf:Stacey03}
A.M.~Stacey, {Branching random walks on quasi-transitive graphs},
Combin.~Probab.~Comput.~\textbf{12}, (2003), n.3
 345--358.

\bibitem{cf:Woess}
        W.~Woess, {Random walks on infinite graphs and groups},
        Cambridge Tracts in Mathematics, {\textbf 138},
    Cambridge Univ.~Press, 2000.

\bibitem{cf:Woess09}
        W.~Woess,
Denumerable Markov chains,
Generating functions, boundary theory, random walks on trees.
EMS Textbooks in Mathematics,
European Mathematical Society (EMS), 2009.
\end{thebibliography}
\end{document}